\documentclass{article} 

\usepackage{etex}
\usepackage[utf8]{inputenc} 
\usepackage[english]{babel} 
\usepackage{anysize} 

\usepackage{amsthm}  
\usepackage{amsmath}
\usepackage{amsfonts}
\usepackage{amssymb}

\usepackage{multicol} 
\usepackage{graphicx} 
\usepackage[dvipsnames]{xcolor} 
\usepackage{pgf,tikz}
\usepackage{tikz-cd}
\usepackage{pictex}
\usepackage{mathrsfs}
\usepackage{enumerate}
\usepackage{verbatim} 
\usepackage{array}
\usepackage[hidelinks]{hyperref}
\usepackage{tikz}
\usepackage{centernot}

\usepackage{tikz-cd}									    					

\marginsize{2 cm}{2 cm}{1 cm}{2 cm}

\newcommand{\f}[5]{
\begin{array}{rrcl}
#1:& #2 & \longrightarrow & #3 \\
& #4 & \longmapsto & #5  \\
\end{array}
}

\numberwithin{equation}{section}
\newtheorem{lemma}[equation]{Lemma}
\newtheorem{proposition}[equation]{Proposition}
\newtheorem{theorem}[equation]{Theorem}

\newtheorem{THEOREM}{Theorem}

\newtheorem{corollary}[equation]{Corollary}
\theoremstyle{definition}
\newtheorem{obs}[equation]{Observation}
\newtheorem{remark}[equation]{Remark}
\theoremstyle{definition}
\newtheorem{example}[equation]{Example}
\newtheorem{definition}[equation]{Definition}
\newtheorem{construction}[equation]{Construction}

\numberwithin{equation}{section}

\newcommand{\reg}{{\mathrm{reg}}}

\newcommand{\gap}[1]{{\color{OliveGreen} \sf $\clubsuit\overset{\bullet\ \cdot}{\frown}\clubsuit$ GAP: [#1]}}

\newcommand{\C}{\mathbb C}
\newcommand{\Z}{\mathbb Z}
\newcommand{\N}{\mathbb N}
\newcommand{\A}{\mathbb A}
\newcommand{\wh}[1]{\widehat{#1}}
\newcommand{\Gm}{\mathbb{G}_m}
\newcommand{\D}{\mathcal{D}_{\Gm}}

\newcommand{\Dz}{\mathcal{D}_{K_0}}
\newcommand{\dd}{\Delta}
\newcommand{\ddA}{\dd_{\A^1}}
\newcommand{\ddAp}{\dd_{\A^{1*}}}
\newcommand{\As}{{\A^{1*}}}
\newcommand{\pz}{{U_p}}
\newcommand{\pzz}{{U_p^*}}
\newcommand{\pzzi}{{U_{p+i}^*}}
\newcommand{\pzi}{{U_{p+i}}}

\DeclareMathOperator{\Mod}{\mathbf{Mod}}

\DeclareMathOperator{\supp}{supp}
\DeclareMathOperator{\Hom}{Hom}
\DeclareMathOperator{\Ext}{Ext}
\DeclareMathOperator{\coker}{coker}
\DeclareMathOperator{\Spec}{Spec}
\renewcommand{\P}{\mathbb{P}}

\newcommand{\Hol}{\mathcal{H}ol}

\newcommand{\M}{\mathcal{M}}
\newcommand{\wt}[1]{\widetilde{#1}}

\newcommand{\jj}{\mathring{\jmath}}
\newcommand{\fin}{\mathrm{fin}}

\DeclareMathOperator{\im}{im}
\DeclareMathOperator{\Id}{Id}

\newcommand{\Loc}{{\Hol(\pz)}}
\newcommand{\Locc}{{\Hol(\pzz)}}
\newcommand{\ov}[1]{\overline{#1}}

\newcommand{\sstack}[2]{{\substack{#1\\#2}}}
\newcommand{\holstar}{\wh\Hol^{\raisebox{-4 pt}{\scriptsize{*}}}}
\newcommand{\modOrbit}{\Mod(\ddA^l)_{p+\Z}}
\newcommand{\holoOrbit}{\Hol(\ddA^l)_{p+\Z}}
\newcommand{\pmodOrbit}{\wh\Mod(\ddA^l)_{p+\Z}}
\newcommand{\pholoOrbit}{\wh\Hol(\ddA^l)_{p+\Z}}
\newcommand{\ppholoOrbit}{\holstar(\ddA^l)_{p+\Z}}
\newcommand{\finmod}{\Mod(\C[\pi])_{\mathrm{fin}}}
\newcommand{\finmodd}{\Mod(\C[z-p])_{\mathrm{fin}}}
\newcommand{\Mel}{\M^{(0,p+\Z)}}
\newcommand{\Gl}{\mathcal{S}}
\DeclareMathOperator{\an}{{an}}

\newcommand{\cY}{{\mathcal Y}}
\newcommand{\cH}{{\mathcal H}}
\newcommand{\bH}{{\mathbb H}}
\DeclareMathOperator{\GL}{{GL}}
\newcommand{\matriz}[4]{\left(\begin{matrix} #1&#2\\#3&#4\end{matrix}\right)}

\title{The local information of difference equations}
\author{Mois\'es Herrad\'on Cueto}
\date{}

\begin{document}

\maketitle

\tableofcontents

\abstract{We give a definition for the restriction of a difference module on the affine line to a formal neighborhood of an orbit, trying to mimic the analogous definition and properties for a $D$-module. We show that this definition is reasonable in two ways. First, we show that specifying a difference module on the affine line is equivalent to giving its restriction to the complement of an orbit, together with its restriction to a neighborhood of an orbit and an isomorphism between the restriction of both to the intersection. We also give a definition for vanishing cycles of a difference module and define a local Mellin transform, which is an equivalence between vanishing cycles of a difference module and nearby cycles of its Mellin transform, a $D$-module.}

\section{Statements}

This paper concerns algebraic difference equations on the affine line and their singularities. Analogously to the study of differential equations using $D$-modules, we use difference modules to study these difference equations. Difference modules are modules over the ring of difference operators, i.e. the ring generated by multiplication by functions and translations. This is the ring $\ddA = \C[z]\langle \tau,\tau^{-1}\rangle$ given by the relations $\tau z=(z-1)\tau$ and $\tau\tau^{-1}=\tau^{-1}\tau=1$. We interpret $\tau$ as the translation $z\mapsto z+1$. 



The most familiar approach to the study of linear algebraic difference equations is to study matrix difference equations: given $A(z)\in \GL_n(\C(z))$, one has the system of equations $y(z+1) = A(z)y(z)$. Two systems given by matrices $A_1,A_2$ are called gauge equivalent if there exists some $R(z)\in \GL_n(\C(z))$ such that $A_2(z) = R(z+1)^{-1}A_1(z)R(z)$. Studying matrix difference equations up to gauge equivalence is equivalent to studying $\C(z)\langle\tau,\tau^{-1}\rangle$-modules which are finite dimensional over $\C(z)$ (see Construction~\ref{con:diffEq}).

The relation between matrix difference equations and $\ddA$-modules is the same as the relation between generic local systems and holonomic $D$-modules: every matrix difference equation has a canonical ``intermediate extension'' (Construction~\ref{Con:GMExtension}), which is a holonomic $\ddA$-module, and every holonomic $\ddA$-module generically becomes a matrix difference equation (Proposition~\ref{prop:finiteStalks}). Part of our motivation to study $\ddA$-modules is that, just like $D$-modules, they have better functorial properties, which are essential for the results of this paper. From a more geometric point of view, they are quasicoherent $\Z$-equivariant sheaves on $\A^1$.

Our goal is to give a good definition for the formal local type of a $\ddA$-module at a point $p\in \A^1$. For matrix difference equations, we have the \textit{monodromy matrix} from \cite{AB}, which can be recovered from our definition (Corollary~\ref{cor:relationToBorodin}). We also show how our local type can be computed from the zeroes and poles of a matrix defining a difference equation (Proposition~\ref{prop:matrices}), which are what one naively would define as ``singularities'' of the equation.

In the remainder of the paper, we show that this formal type $M\mapsto M|_\pz$ has two desirable properties. First, we show that a module satisfies a ``sheaf'' property under this definition: it can be uniquely (and functorially) recovered from the data of $M|_\pz$, $M|_{\A^1\setminus p}$ and an isomorphism between these two defined on $U_p\setminus p$. This theorem ensures that the formal type doesn't lose information. As an application, we describe how to obtain all the $\ddA$-modules which are generically equal to a given matrix difference equation (Section~\ref{sec:extendOverPuncture}).

Secondly, we show that this construction is compatible with the Mellin transform. The Mellin transform is a particular case of the Fourier transform of \cite{L}, which can be seen as the ring isomorphism $\D:=\C[z]\langle \tau,\tau^{-1}\rangle \cong \ddA$ mapping $x$ to $\tau$ and $x\partial_x$ to $z$. We show that the local types of a $D$-module at $0$ and $\infty$ predict the local types of its Mellin transform at all points $p\in \A^1$, via an equivalence we call the local Mellin transform. This statement is to be expected, given that for the usual Fourier transform, which maps $D$-modules on $\A^1$ to themselves, one has local Fourier transforms \cite{BE}, and in loc. cit. it is shown that collection of local types of a $D$-module  at every point $p\in \A^1\cup\infty$ completely determine the local types of its Fourier transform. The result in this paper, together with the result in \cite{GS}, shows that the collection of local types of a $D$-module on $\A^1\setminus 0$ completely determine, and are determined by, the local types of its Mellin transform. Both in the Fourier and the Mellin cases, to make the statement precise one needs to be careful with what is meant by ``local type'' (see \cite{A} for the Fourier transform). In our situation, we define functors for difference modules which we believe deserve to be called ``vanishing cycles'' by analogy with the $D$-module case, especially the local Fourier transform as stated in loc. cit.

Note that in order to talk about the Mellin transform, it is necessary to use $\ddA$-modules rather than matrix difference equations. The intermediate extensions allows matrix difference equations to be mapped to $\ddA$-modules, but then one runs into the issue that the Mellin transform of an intermediate extension is not necessarily an intermediate extension, and the same happens for the inverse Mellin transform.

Algebraic difference equations are of interest to the study of spaces of initial conditions of Painlev\'e equations, especially discrete Painlev\'e equations. In \cite{AB}, it is shown how some discrete and differential Painlev\'e equations arise as isomonodromy transformations and deformations respectively, for certain moduli spaces of difference equations, namely on spaces of difference equations with a certain specified local type. Moreover, in loc. cit. an example of the local Mellin transform is shown, as a moduli space of difference equations with a given local type is shown to be isomorphic via the Mellin transform to a space of differential equations with another given local type. The results in this paper provide a framework in which such a construction can be done in general.

Other Painlev\'e equations arise as isomonodromy transformations of different discrete equations, such as $q$-difference equations and elliptic difference equations. In \cite{h2}, we show how one can build on these ideas to define the local type for all of these.

\subsection{The local type}

Let us start by comparing our situation to the $D$-module case. A $D$-module on the affine line is a module over the ring of differential operators $\C[x]\langle \partial_x\rangle$, where $\partial_x x=x\partial_x+1$. Given such a $D$-module $M$, it can be restricted to a formal disk around a point $p$: this can be thought of as the functor $M\mapsto M_p:=\C[[x-p]]\otimes _{\C[x]} M$, where $M_p$ has a natural $\C[[x-p]]\langle \partial_x\rangle$-module structure. The analogy with difference equations starts to break down here: if $M$ is a $\ddA$-module, there is no reasonable way to endow $M_p=\C[[z-p]]\otimes M$ with an action of $\tau$, but rather $\tau$ identifies $M_p$ with $M_{p+1}$.

If we restrict ourselves to difference modules which are ``small'' (i.e. holonomic, see Definition~\ref{def:Holonomic}), then we can easily see that their stalks are finitely generated (Proposition~\ref{prop:finiteStalks}). Therefore, $M_p$ is a finitely generated module over $\C[[z-p]]$, so it is completely described by its rank and its torsion. However, difference modules have singularities that are not captured by this picture. Considering a difference equation for a matrix $A$, zeroes and poles of $A$ should be singularities of the equation. Here by poles (resp. zeroes) we mean points where $A$ (resp. $A^{-1}$) is not defined. However, $M_p$ only remembers the dimension of $A$.




To define the local type, we attach additional structure to $\C[[z-p]]\otimes M$, namely the data of two submodules, which record the information ``at $+\infty$" and ``at $-\infty$".

\begin{construction}
Let $M$ be a holonomic $\ddA$-module, and let $p\in \A^1$. Choose any finitely generated sub-$\C[z]$-module $L\subset M$ such that $M/L$ is a torsion $\C[z]$-module. Then when $n\in \Z$ is big enough, $(\tau^{n} L)_p$ and $(\tau^{-n}L)_p$ are submodules of $M_p$ which are independent of the choices of $L$ and $n$.  The restriction of $M$ to the formal neighborhood $\pz$ of $p+\Z$, denoted $M|_{\pz}$, is defined to be the module $M_p$,  together with the data of two submodules $M|_\pz^l := (\tau^nL)_p$ and $M|_\pz^r := (\tau^{-n}L)_p$.
\end{construction} 

Thus the restriction $M|_\pz$ is not a difference module, but a $\C[[z-p]]$-module equipped with additional information. The local type lands in the category of diagrams $a\to b\gets c$ of $\C[[z-p]]$-modules. Even though the definition a priori involves an unknown big enough $n$, the local type can be computed in a straightforward way (Proposition~\ref{prop:matrices}).

In Section \ref{sec:restrictionToDisks} we prove that the above definition is well-defined and we show that for the ``open covering'' $\A^1 = U_p \cup (\A^1\setminus p)$, difference modules behave like sheaves. Here is what we mean precisely.

Localizing the ring $\ddA$ at an orbit $p+\Z$ yields the ring \(\ddAp = \C\left[z,\left\{\frac{1}{z-p-n}\right\}_{\raisebox{3pt}{$\scriptscriptstyle{n\in \Z}$}}\right]\langle \tau,\tau^{-1}\rangle\), which gives rise to a similar theory of difference modules on the punctured affine line (described in Section \ref{sec:ddAp}), including an analogous restriction functor $|_\pzz$ from $\ddAp$-modules to $\C((z-p))$-modules equipped with the data of two submodules. This gives rise to a commutative square of restrictions (any commutative diagram of categories in the present paper should be understood as commuting up to natural isomorphism):
\begin{equation}\label{eq:fiberDiagramIntro}
\begin{tikzcd}[column sep = 6 em,ampersand replacement=\&]
\Hol(\ddA) \arrow[r,"|_{\As}"]\arrow[d,"|_{\pz}"] \& \Hol(\ddAp)\arrow[d,"|_{\pzz}"] \\
\Loc \arrow[r,"|_{\pzz}=\cdot \otimes \C((z-p))"] \& \Locc.\\
\end{tikzcd}
\end{equation}
\vspace{-2.5em}

Here $\Hol(\ddA)$ denotes holonomic difference modules on the affine line, and analogously for $\Hol(\ddAp)$. These definitions are given in Section \ref{sec:background}. The remaining arrows and categories are all defined in Section \ref{subsec:RestrDefs}.

\begin{THEOREM}\label{thm:formalCycles}
The diagram (\ref{eq:fiberDiagramIntro}) is a fibered product of categories.
\end{THEOREM}

In other words, holonomic difference modules can be recovered from their restrictions to $\A^1\setminus (p+\Z)$ and $\pz$, together with a compatibility between these two, which amounts to a given isomorphism between their restrictions to the punctured disk $\pzz$. Conversely, this information is enough to determine a difference module on the line.

This is analogous to the fact that functions on a scheme form a sheaf, or, in a slightly different way, can be seen as analogous to the following statement, which is an example of faithfully flat descent. Let $V$ be a variety, let $p\in  V$, and let ${\mathcal O_p}$ be its completed local ring. Then there is a commutative square, all of whose arrows are pullbacks:
\[
\begin{tikzcd}[column sep = 5 em,ampersand replacement=\&]
\operatorname{QCoh}(V) \arrow[r]\arrow[d] \& \operatorname{QCoh}(V\setminus \{p\})\arrow[d] \\
\Mod({\mathcal O_p}) \arrow[r] \& \Mod\left(\operatorname{Frac}({\mathcal O_p})\right).\\
\end{tikzcd}
\]
\vspace{-2.5em}

This square is a Cartesian square of categories. It is not even necessary to complete the local ring: the statement would be true replacing ${\mathcal O_p}$ by the stalk of $\mathcal O$ at $p$. The same is true for the diagram (\ref{eq:fiberDiagramIntro}): We could replace formal fibers everywhere by stalks and obtain the same statement (Proposition~\ref{prop:stalksInsteadofFibers}).

\subsection{Vanishing cycles and Mellin transform}


%

For every pair of points $p\in \Gm\cup \{0,\infty\}=\P^1_x$ and $q\in \A^1\cup \{\infty\}=\P^1_z$, there is a local Mellin transform $\M^{(p,q)}$ relating the local type of $D$-modules at $p$ with the local type of difference modules at $q$. There are four essentially different behaviors depending on whether $p$ and $q$ are points ``at infinity'' (including $p=0$). If neither $p$ or $q$ is at infinity, then the local Mellin transform is trivial: it should be the functor between two categories which are $0$ (it can be seen that the local type at finite points has no influence on the local type at finite points, from the fact that all the other local Mellin transforms completely determine the local types). If $q=\infty$, then there are two possibilities: either $p\in \Gm$ or $p=0,\infty$. Both of these possibilities were defined and computed in \cite{GS}. Lastly, there are local Mellin transforms from $0$ or $\infty$ in $\P^1$ to points $p\in \A^1$, which are the focus of the present paper. Actually, translation makes all Mellin transforms on the same $\Z$-orbit isomorphic, and consequently we will denote these local Mellin transforms $\Mel$ and $\M^{(\infty,p+\Z)}$.

Following the analogy with the local Fourier transform, the local Mellin transform $\Mel$ should give an equivalence between some nearby cycles of a $D$-module $M$ at $0$ and vanishing cycles of $\M (M)$ at $p+\Z$. So first of all we need a notion of vanishing cycles for a difference module.

\begin{definition}\label{def:vanishingCycles}
Let $M\in \Hol(\ddA)$ and $p\in \A^1$. The \textbf{left (resp. right) vanishing cycles} of $M$ at $p+\Z$ are defined as
\[
\Phi_{p+\Z}^lM : =\frac{M_p}{M|_\pz^l} ,\hspace{4em} \text{(resp.)}\hspace{1em}\Phi_{p+\Z}^rM : =\frac{M_p}{M|_\pz^r}.
\]
\end{definition}
These functors approximately compute a familiar notion: if a difference equation is given as a matrix difference equation $y(z+1)=A(z)y(z)$, where $A(z)\in \GL_n(\C(z))$, then $\Phi_{p+\Z}^r$ computes the poles of the matrix $A$ and $\Phi_{p+\Z}^l$ computes the zeroes. This is not completely true, since taking a gauge transformation $y(z)=R(z)\wt y(z)$ for $R(z)\in \GL_n(\C(z))$ might change the set of zeroes and poles, in two ways. First, the zeroes and poles could be translated in the same $\Z$-orbit (consider for example the change $R(z)=A(z)^{-1}$). To deal with this, we think of singularities as a property of the whole $\Z$-orbit. Second, a gauge change might introduce new \emph{apparent} singularities (``apparent'' because they can be removed by a gauge change). The local type and vanishing cycles provide a notion of zeroes and poles that is coordinate independent, and it can detect apparent singularities (Remark~\ref{rem:apparent}). Proposition~\ref{prop:matrices} describes the exact relation between the local type and a matrix difference equation.

Note that $\Phi^l_{p+\Z}$ and $\Phi^r_{-p+\Z}$ can be interchanged by the automorphism $z\leftrightarrow -z$ and $\tau\leftrightarrow \tau^{-1}$ on the difference side, and $x\leftrightarrow x^{-1}$ on the $D$-module side, thus we may just focus on $\Phi^l$ from now on. We show that the image of $\Phi^l_{p+\Z}$ lands in the category $\finmodd$ of finite length $\C[[z-p]]$-modules. Further, we construct its right adjoint, which we denote $\iota_{p!}^\to$. It is simply the functor $N\mapsto \C((\tau))\otimes_\C N$.

These are all the necessary ingredients to construct the local Mellin transform. On the $D$-module side, we should just focus on regular $D$-modules, and this can be seen just from the fact that irregular $D$-modules are the input of $\M^{(0,\infty)}$. Thus, we will let $\Hol(\Dz)^{\reg}$ denote the category of regular holonomic modules over $\Dz=\C((x))\langle \partial_x\rangle$. Following \cite{A}, we denote by $\jj_{0*}$ the forgetful functor from $\C((x))\langle\partial_x\rangle$-modules to $\C[x,x^{-1}]\langle \partial_x\rangle$-modules.

Further, using the classification of $D$-modules over the formal disk (originally proved by Turrittin \cite{Turrittin} and Levelt \cite{levelt1975}, but the proof can also be found in \cite{van2003galois}), we may split them according to the leading term, which is well defined up to adding an integer. We denote $\Hol(\Dz)^{\reg,(p)}$ the category of those regular $D$-modules whose leading term is in $p+\Z$.

\begin{THEOREM}\label{thm:localMellin}\leavevmode
\begin{enumerate}[1)]
\item For any $p+\Z\in \A^1/\Z$, there is an equivalence
\[\Mel : \Hol (\mathcal{D}_{K_0} )^{\reg,(p)} \longrightarrow \finmodd.\]
and for any $F\in \Hol (\mathcal{D}_{K_0} )^{\reg,(p)}$, there is a functorial isomorphism
\[
\M(\jj_{0*} (F))\overset{\sim}\longrightarrow \iota_{p!}^\to (\Mel(F)).
\]The isomorphism is a homeomorphism in the natural topology, i.e. the $\tau$-adic topology. This determines $\Mel$ up to natural isomorphism.
\item For any $p+\Z\in \A^1/\Z$, there is an equivalence
\[\M^{(\infty, p+\Z)} : \Hol (\mathcal{D}_{K_\infty})^{\reg,(-p)} \longrightarrow \finmodd .\]
and for any $F\in \Hol (\mathcal{D}_{K_\infty} )^{\reg,(-p)}$, there is a functorial isomorphism
\[
\M(\jj_{\infty*} (F))\overset{\sim}\longrightarrow \iota_{p!}^\gets (\M^{(\infty,p+\Z)}(F)).
\]The isomorphism is a homeomorphism in the natural topology, i.e. the $\tau$-adic topology. This determines $\M^{(\infty,p+\Z)}$ up to natural isomorphism.
\end{enumerate}
\end{THEOREM}

Analogously to the result in \cite{A}, the adjunctions $\Phi^l_{p+\Z}\vdash \iota_{p!}^\to$, $\Psi_0\vdash \jj_{0*}$ (where $\Psi$ denotes nearby cycles) immediately yield the following corollary, which as desired gives the relation between the local information of a $D$-module and that of its Mellin transform.
\begin{corollary}
Let $F\in \Hol(\D)$. For any $p\in \A^1/\Z$, there are natural isomorphisms
\begin{align*}
\Phi_{p+\Z}^l(\M(F)) &\cong \Mel(\Psi_0(F)^{\reg,(p)});\\
\Phi_{p+\Z}^r(\M(F)) &\cong \M^{(\infty,p+\Z)}(\Psi_\infty(F)^{\reg,(-p)}).
\end{align*}
\end{corollary}
Here we write $M^{\reg,(p)}$ to denote the functor that picks from a $\Dz$-module its regular singular summand with leading coefficient $p$ (i.e. the left adjoint to the inclusion of these submodules into general $\Dz$-modules).

\subsubsection*{Acknowledgments}

I am very grateful to Dima Arinkin for suggesting the problem and for much help along the way. I also wish to thank Juliette Bruce, Eva Elduque and Solly Parenti  for many useful discussions. This work was partially supported by National Science Foundation grant DMS-1603277.

\section{Difference equations and difference modules}\label{sec:background}

A system of linear difference equations is given as follows: let $A(z)\in \GL_n(\C(z))$, and consider the equation
\[
y(z+1)=A(z)y(z).
\]
Where $y$ is a vector function $\A^1\to \C^n$. A natural generalization of this setting, which allows for difference equations to be defined locally, comes from taking $y$ to be not a section of a trivial bundle $\mathcal O^n$, but of a vector bundle $V$ without a trivialization. In this setting, a difference equation is a rational isomorphism $\mathcal{A}:V\overset\sim\dasharrow t^*V$, where $t$ denotes the translation of $\A^1$. In \cite{AB} this is called a d-connection. On the fibers where $\mathcal A$ is defined and invertible, $\mathcal A|_z$ is an isomorphism $V|_z\to V|_{z+1}$ that depends on $z$ as a rational function. This notation will only be used briefly in Section~\ref{sec:relationToBorodin}.

Instead, we can simplify our notation using the fact that $\A^1$ is affine. We can identify a vector bundle $V$ with its $\C[z]$-module of global sections. Given such a $\mathcal A$, we consider $\tau = (t^*)^{-1}\circ \mathcal A:V\to V$. Now, since $t^*$ is not $\C[z]$-linear, $\tau$ is not linear, but rather we have the relation $\tau z=(z-1)\tau$. Conversely, if we are given a $\C(z)\langle \tau^{\pm 1}\rangle$-module $V$, we can take the corresponding vector bundle (up to rational isomorphism) and define $\mathcal A = t^*\circ \tau$.

In sight of this description, it seems natural to consider all $\ddA$-modules, regardless of whether they are vector bundles. Let us now fix the notation for the difference module corresponding to a difference equation.

\begin{construction}\label{con:diffEq}
Consider $A\in GL_n(\C(z))$ and the difference equation $y(z+1)=A(z)y(z)$. Make $\C(z)^n$ into a difference module by letting $\tau(y(z))=A(z-1)y(z-1)$. This ensures that $\tau (z y(z))=(z-1)\tau(y(z))$. Inside $\C(z)^n$, consider the trivial vector bundle $L=\C[z]^n\subset\C(z)^n$. The \textbf{difference module corresponding to $A$} is the $\ddA$-module $M$ generated by $L$. The solutions to the difference equation are ``horizontal'' sections, i.e. sections which are fixed by $\tau$.
\end{construction}

\begin{remark}
A gauge transformation $y(z)=R(z)\wt y(z)$ for $R(z)\in \GL_n(\C(z))$ yields an equivalent difference equation, with matrix $\wt A(z) = R(z+1)A(z)R(z)^{-1}$. The corresponding difference module will be generated by the columns of $R(z)$, i.e. by a vector bundle $\wt L$ which is a modification of $L$ at the points where $R$ or $R^{-1}$ is not defined. Over $\ddA$, it is possible for $L$ and $\wt L$ to generate the same module even if they are not equal. For the simplest example, consider $A(z)=(1)$, and $R(z)=(z)$.
\end{remark}


Given a d-connection, i.e. a module over $\C(z)\langle \tau^{-\pm 1}\rangle$ which is a finite dimensional vector space over $\C(z)$, there are many choices of finitely generated $\ddA$-modules which generate it. There is, however, a smallest such one, which we call the intermediate extension by analogy with the $D$-module case. We discuss its properties in Section~\ref{sec:intermediateExtension}. Intermediate extensions can be recognized by the local type (see Remark~\ref{rem:GMextension}), and Section~\ref{sec:extendOverPuncture} details how to use the local type to describe all holonomic $\ddA$-modules which generically look the same. If one is only looking for $\ddA$-modules which are torsion-free, the answer (at every orbit $p+\Z$) is identical to the question of extending a torsion $\C[[z-p]]$-module by a free $\C[[z-p]]$-module to obtain a torsion-free module (Corollary~\ref{cor:extensions}).




\subsection{Holonomic difference modules}

We will restrict our attention to holonomic difference modules. Over the affine line, they have an analogous definition to the one for $D$-modules.

\begin{definition}\label{def:Holonomic}
A $\ddA$-module is \textbf{holonomic} if it is finitely generated over $\ddA$ and every element is annihilated by a nonzero element of $\ddA$.
\end{definition}

The same definition holds for a $D$-module (see for example \cite[Chapter 10]{coutinho}), so a difference module is holonomic if and only if its Mellin transform is holonomic. In general, holonomic difference modules are characterized by the codimension of their singular support, see \cite{MP}. Note that the Mellin transform shows that holonomic difference modules have finite length, since holonomic $D$-modules do.

Holonomic difference modules satisfy many desirable properties. However, they are not vector bundles over a dense open set, as in the case of $D$-modules. The simplest counterexample is the torsion module $\delta_0$ generated by an element $s$ and the relation $zs=0$. The element $\tau^ns$ is supported on $n\in \A^1$, which yields a countable collection of points where  $\delta_0$ has torsion.

One desirable property, however, is the following, which does not hold for holonomic $D$-modules.

\begin{proposition}\label{prop:finiteStalks}
Any holonomic $\ddA$-module has finite stalks, i.e. if $M\in Hol(\ddA)$, then for any $p$, $\C[z]_{(z-p)}\otimes_{\C[z]} M$ is a finitely generated $\C[z]_{(z-p)}$-module.
\end{proposition}
\begin{proof}
Let us first prove it for a cyclic $\ddA$-module (actually all holonomic modules are cyclic, but we will not need this fact here). Let $M$ be generated by an element $s$. If $M$ is holonomic, then $s$ is annihilated by a nonzero element of $Q\in \ddA$, which after multiplying by a suitable power of $\tau$ can be written as $Q =\sum_{i=0}^{n} P_i(z) \tau^i$, where $P_n,P_0\neq 0$ and $n\ge 0$. Choose some $N\gg 0$ such that $(z-p)\centernot| P_0(z-m)P_n(z-m)$ whenever $|m|\ge N$. Then $P_0(z-m)$ and $P_n(z-m)$ are invertible in $\C[z]_{(z-p)}$ and the following identities hold:
\begin{align*}
\tau^{m+n}s &=\frac{-1}{P_n(z-m)} \sum_{i=0}^{n-1} P_i(z-m) \tau^{i+m}s; & 
\tau^{-m}s &=\frac{-1}{P_0(z+m)} \sum_{i=1}^{n} P_i(z+m) \tau^{i-m}s.
\end{align*}
This implies that the finite set $\{\tau^{1-N}s,\ldots ,\tau^{N+n-1}s\}$ generates the stalk $\C[z]_{(z-p)}\otimes_{\C[z]} M$ over $\C[z]_{(z-p)}$. Note that the proof holds if $n=0$, in which case the sums above become empty.

If $M$ is not cyclic, then the statement follows by induction on the length.
\end{proof}

\begin{corollary}\label{finiteRank}
A holonomic $\ddA$-module has finite generic rank, i.e. if $M\in \Hol(\ddA)$, $\C(z)\otimes_{\C[z]} M$ is finite dimensional over $\C(z)$.
\end{corollary}

\begin{remark}
If a module $M$ is finitely generated over $\ddA$, then the converse to the corollary above is also true. If $M$ has finite generic rank, then it cannot contain $\ddA\cong \bigoplus_{i\in \Z}  \C[z]\tau^i$ as a submodule, because it has infinite rank, and therefore every element is torsion.
\end{remark}

\subsection{Difference modules on the punctured affine line}\label{sec:ddAp}
One of our goals is to relate difference modules on the affine line to difference modules on the punctured affine line. Due to the action of $\Z$, instead of removing a single point from the line, one must remove a whole $\Z$-orbit $p+\Z$. In the sequel, we will let $p\in \A^1$ be fixed, and we will let $\As = \A^1\setminus (p+\Z) = \Spec \C\left[z,\frac{1}{z-p},\frac{1}{z-p\pm 1},\ldots 
\right]$.
\begin{definition}Difference modules on the punctured line $\A^1$ are defined to be left modules over the ring
\[
\ddAp = \C\left[ z,\left\{\frac{1}{z-p-i}\right\}_{i\in \Z} \right]\langle \tau,\tau^{-1}\rangle.
\]
A difference module is said to be \textbf{holonomic} if it is finitely generated and every element is annihilated by a nonzero $f\in \ddAp$. We denote the categories of difference modules and holonomic modules on the punctured line by $\Mod(\ddAp)$ and $\Hol(\ddAp)$, respectively.
\end{definition}

\begin{definition}
Define the \textbf{restriction to the punctured line} functor $|_\As:\Mod(\ddA)\longrightarrow \Mod(\ddAp)$ as follows: for $M\in \Hol(\ddA)$,
\[
M|_\As = \C\left[z,\frac{1}{z-p}\right]\otimes_{\C[z]} M.
\]
The action of $\tau$ on $M|_\As$ is defined by $\tau(P(z)\otimes m)= P(z-1)\otimes \tau m$. It makes $M|_\As$ into a left $\ddAp$-module. Note that $(z-p-i)$ acts as a unit on $M|_\As$ for any $i\in \Z$, since $(z-p-i)=\tau^i(z-p)\tau^{-i}$.
\end{definition}


\begin{remark}
If $M$ is holonomic, then $M|_\As$ is holonomic, so $|_\As$ gives rise to a functor $|_\As:\Hol(\ddA)\longrightarrow \Hol(\ddAp)$.
\end{remark}

\subsection{Coherent subsheaves of difference modules}

For a holonomic difference module, we will repeatedly make use of its finitely generated $\C[z]$-submodules, particularly the ones that are generically equal to the given module.

\begin{definition}\label{Def:NotReallyLattices}
Let $M\in \Mod(\ddA)$. We define the set $\Gl(M)$ to be the set of $\C[z]$-submodules $L\subseteq M$ such that $L$ is finitely generated over $\C[z]$ and $M/L$ is a torsion $\C[z]$-module.
\end{definition}

\begin{obs}\label{obs:glattices}
$\Gl(M)$ is nonempty whenever $M\in \Hol(\ddA)$ or $\Hol(\ddAp)$, by Corollary \ref{finiteRank}, or when $M$ is a finitely generated $\C(z)$-vector space. If $\Gl(M)\neq \emptyset$, every finite subset of $M$ is contained in some $L\in \Gl(M)$.
\end{obs}

\begin{definition}
Let $M\in \Hol(\ddA)$ and let $L\in \Gl(M)$. We define the \textbf{zeroes} of $L$ as the finite set $Z_L= \supp \displaystyle\frac{\tau^{-1}L}{L\cap \tau^{-1} L} = \supp \displaystyle\frac{L +\tau^{-1} L}{L}$, and the \textbf{poles} of $L$ as the finite set $P_L = \supp \displaystyle\frac{L+\tau^{-1} L}{\tau^{-1}L} = \supp \displaystyle\frac{ L}{L\cap \tau^{-1} L}$.
\end{definition}

\begin{remark}
If $L$ and $M$ come from a matrix $A$ via Construction~\ref{con:diffEq}, then $P_L$ is the set of poles of $A$ and $Z_L$ is the set of zeroes. Both of these sets are finite because they are the supports of finitely generated torsion $\C[z]$-modules.
\end{remark}

\begin{lemma}\label{lem:glattices}
Let $M\in \Hol(\ddA)$ and $L\in \Gl(M)$. Then $M/L$ is supported on finitely many orbits.
\end{lemma}

\begin{proof}
Two finitely generated modules that agree over $\C(z)$ are equal away from a finite set. Therefore, it is enough to prove the statement for any $L\in \Gl(M)$.

Let $L\in \Gl(M)$ be chosen such that it contains a finite generating set of $M$ over $\ddA$. We will prove by induction that $L_n=(L+\tau L+\cdots +\tau^n L)/L$ is supported on $Z_L+P_L+\Z$ (actually on $P_L+\Z_{\ge 0}$). We use the following short exact sequence, together with the fact that the support of a module is contained in the union of the supports of a submodule and quotient:
\[0\longrightarrow \frac{L+\cdots +\tau^{n-1} L}{L }\longrightarrow \frac{L+\cdots +\tau^n L}{L }\longrightarrow \frac{L+\cdots +\tau^n L}{L +\cdots \tau^{n-1} L}\longrightarrow 0.
\]
We note that $\frac{\tau^{n-1}L+\tau^nL}{\tau^{n-1}L}$ surjects onto $\frac{L+\cdots +\tau^n L}{L +\cdots +\tau^{n-1} L}$
, so the support on the latter is contained in the support of the former. Further, we note that $\frac{\tau^{n-1}L+\tau^nL}{\tau^{n-1}L} = \tau^{n} \frac{L+\tau^{-1} L}{\tau^{-1} L}$, which implies that $\supp \frac{\tau^{n-1}L+\tau^nL}{\tau^{n-1}L}= P_L+n$. The induction hypothesis then shows that $\supp \frac{L+\cdots +\cdots +\tau^n L}{L}\subset P_L+\Z_{\ge 0}$.

The analogous reasoning yields the same result for negative $n$'s, and these together show the desired result, since $M=\sum_{n\in \Z} \tau^n L$.
\end{proof}

\subsection{The intermediate extension}\label{sec:intermediateExtension}

One of the first questions one can ask is whether any d-connection, or more generally any holonomic $\ddAp$-module can be extended over the puncture to a $\ddA$-module in some canonical way. For a $D$-module, there are three answers, namely $j_*$, $j_!$ and $j_{!*}$, whose definitions can be found in \cite{takeuchi2007}, for example.


 
For difference modules, we have $j_*$, the forgetful functor, which has the disadvantage that it does not preserve holonomic modules. However, the intermediate extension $j_{!*}$ does have a difference analogue, which preserves holonomicity. It can be constructed as the smallest $\ddA$-module contained in a given $\ddAp$-module that only differs from it at $p+\Z$.

\begin{construction}\label{Con:GMExtension}
Let $M\in \Hol(\ddAp)$. The intermediate extension of $M$, denoted $j_{!*}M$ or $j_{!*}^{p+\Z}M$, is constructed as follows: Consider some $L\in \Gl(M)$ such that $P_L \cap (Z_L + \Z_{> 0}) = \emptyset$. We will call such submodules \textbf{austere}\footnote{Because they dispense with unnecessary elements.}.

Then $j_{!*}M$ is defined as the $\ddA$-submodule of $M$ generated by all the subspaces $P(z)^{-1}L$ for all polynomials $P(z)$ with no roots in $p+\Z$, and $L\in \Gl(M)$ can be arbitrary provided it is austere. In other words, if $L$ is chosen to agree with $M$ outside of $p+\Z$, then $L$ generates $j_{!*}M$.
\end{construction}

\begin{proposition}\label{Prop:GMExtensionWellDef}\leavevmode\begin{enumerate}
\item Any holonomic module $M$ contains austere submodules $L\in \Gl(M)$. Furthermore, any $W\in \Gl(M)$ contains a submodule $L\in \Gl(M)$ that is austere.
\item Any two austere submodules $V\in \Gl(M)$ generate the same module $j_{!*}M$ by Construction \ref{Con:GMExtension}.
\end{enumerate}
\end{proposition}
\begin{proof}
\begin{enumerate}
\item Consider any submodule $W\in \Gl( M)$. We claim that for the submodule $W' = \tau^{-1} W\cap W\in \Gl(M)$, its poles satisfy $P_{W'}\subseteq P_W-1$: Indeed
\[
\frac{W'}{\tau^{-1} W'\cap W'} = \frac{\tau^{-1}W\cap  W}{\tau^{-2} W \cap \tau^{-1}W\cap  W} \subseteq \frac{\tau^{-1}W}{\tau^{-2} W\cap \tau^{-1}W}.
\]
And further $\supp \frac{\tau^{-1}W}{\tau^{-2} W\cap \tau^{-1}W} =-1+\supp \tau\left( \frac{\tau^{-1}W}{\tau^{-2} W\cap \tau^{-1}W}\right) =
P_W-1$. Turning to the set of zeroes, we see that $Z_W'\subseteq Z_W$, since
\[
\frac{W'}{\tau^{-1} W'\cap  W'} = \frac{\tau^{-1}W\cap W}{\tau^{-2}W\cap \tau^{-1} W\cap  W}\subseteq \frac{\tau^{-1} W}{\tau^{-1} W\cap  W}.
\]
Thus iterating this process shifts the poles of the submodule to the left, while the zeroes do not move at all. Analogously, considering the submodule $W'' = \tau W\cap W$, one can check that $P_{W''}\subseteq P_W$ and $Z_{W''}\subseteq Z_W +1$, which allows to move the zeroes to the right while keeping the poles in place. This process of shifting the zeroes and poles must reach a submodule $V$ which is austere.

\item Let $L,L'\in \Gl(M)$ be austere. By virtue of the first part of this proposition, without loss of generality we may assume that $L\subset L'$ (by choosing a third submodule in $\Gl(M)$ that is contained in $L\cap L'$). We may also assume that $L'/L$ is supported on $p+\Z$, since modifying $L$ away from $p+\Z$ doesn't affect the construction of $j_{!*}$. For the time being, we will let $j_{!*}M$ be the module generated from $L$ by the procedure above. We will show that $L'\subseteq j_{!*}M$.

We may take the quotient by the $\ddA$-module $j_{!*}M$. Then, $(L'+j_{!*}M)/j_{!*}M$ is a finite dimensional module supported on $p+\Z$. If $(L'+j_{!*}M)/j_{!*}M$ is nonzero, it has some pole to the right of some zero: if the point in the support of $(L'+j_{!*}M)/j_{!*}M$ with the biggest (resp. smallest) real part is $z_1$ (resp. $z_0$), then $z_1$ is a pole and $z_0-1$ is a zero. This contradicts the assumption that $L'$, and therefore $(L'+j_{!*}M)/j_{!*}M$, is austere.

\end{enumerate}

\end{proof}


\begin{proposition}\label{Prop:GMExtensionProps}
Let $M\in \Hol(\ddAp)$. Then the following hold.
\begin{enumerate}

\item $(j_{!*}M)|_\As = M$.

\item The intermediate extension has no nonzero submodules or quotient modules with support contained in $p+\Z$.

\item Out of the modules contained in $M$, $j_{!*}M$ is the smallest $\ddA$-module $N$ such that $N|_\As = M$.

\item If $N$ is a $\ddA$-module such that $N|_\As\cong M$ and $N$ has no nonzero submodules or quotients supported on $p+\Z$, then the map $\phi:N\to N|_\As \to M$ factors through an isomorphism $N\to j_{!*}M$.

\item The functor $j_{!*}$ is fully faithful.
\end{enumerate}
\end{proposition}
\begin{proof}
\begin{enumerate}
\item This follows from the construction, since the stalks of $M$ and $j_{!*}M$ are equal away from $p+\Z$.

\item Since $j_{!*}M\subset M$, $j_{!*}M$ has no elements supported on $p+\Z$. Now suppose that $N\subset j_{!*}M$ is a $\C[z]\langle \tau,\tau^{-1}\rangle$-submodule such that $j_{!*}M/N$ is supported on $p+\Z$. By Proposition \ref{Prop:GMExtensionWellDef}, there is an austere $L\in \Gl(N)$, in particular $L\subseteq N$ and $L\in \Gl(M)$. By the second part of said proposition, $L$ generates all the stalks of $j_{!*}M$ at the points of $p+\Z$, and therefore $N=j_{!*}M$.
\item Let $N$ be such a module. Then $j_{!*}M$ and $N$ coincide outside of $p+\Z$. Therefore, $j_{!*}M/(N\cap j_{!*}M)$ is supported on $p+\Z$, and by the previous part of this proposition, this implies that $j_{!*}M/(N\cap j_{!*}M) =0$, so $j_{!*}M\subseteq N$.
\item First of all, $\phi$ is injective: the kernel of $N\to N|_\As$ is supported on $p+\Z$, so it must vanish. The image of $\phi$ must contain $j_{!*}M$, by the minimality of $j_{!*}M$. Finally, The quotient $N/j_{!*}M$ vanishes after applying $|_\As$, so it is supported on $p+\Z$, so it must vanish by hypothesis.
\item Since $|_\As \circ j_{!*}\cong \Id$, it is enough to show that $|_\As$ is faithful on the essential image of $j_{!*}$. Suppose a map $f:j_{!*}M\to j_{!*}N$ is such that $f|_\As =0$. Then the image of $f$ is a torsion submodule of $j_{!*}N$, which implies that its image is $0$. Therefore, $|_\As$ is faithful.
\end{enumerate}
\end{proof}

\section{Restriction to the formal disk}\label{sec:restrictionToDisks}

Throughout this section, we will let $p\in \A^1$ be fixed, and consider its orbit $p+\Z$. We will show how difference modules on $\A^1$ can be recovered from their restriction to $\A^1\setminus (p+\Z) = \As$ and to a neighborhood of $p+\Z$. For that purpose, we will give a definition for a holonomic difference module on $\bigsqcup_{i\in \Z} \pzi$, where $\pzi= \Spec \C[[z-p-i]]$ is the formal neighborhood of $p+i$, and a similar definition for a difference module on $\bigsqcup_{i\in \Z} \pzzi$, where $\pzzi= \Spec \C((z-p-i))$ is the punctured formal neighborhood of $p+i$. This will yield two categories of difference modules, which we will denote $\Loc$ and $\Locc$, respectively. We will also define restriction functors between these categories, which give rise to a commutative diagram:
\begin{equation}\label{eq:Fibersquare}
\begin{tikzcd}[column sep = 5 em,ampersand replacement=\&]
\Hol(\ddA) \arrow[r,"|_{\As}"]\arrow[d,"|_{\pz}"] \& \Hol(\ddAp)\arrow[d,"|_{\pzz}"] \\
\Loc \arrow[r,"|_{\pzz}"] \& \Locc.\\
\end{tikzcd}
\end{equation}

\vspace{-2.5em}

\setcounter{THEOREM}{0}

The main theorem is Theorem~\ref{thm:formalCycles} from the introduction.
\begin{THEOREM}
The diagram (\ref{eq:Fibersquare}) yields an equivalence $\Hol(\ddA)\cong \Loc\times_\Locc \Hol(\ddAp)$.
\end{THEOREM}
Explicitly, the theorem states that the category $\Hol(\A^1)$ is equivalent to the category of triples $(M_\pz,\linebreak M_\As,\cong)$, where $M_\pz \in \Loc$, $M_\As\in \Hol(\ddAp)$, and $\cong$ is an isomorphism $M_\pz|_\pzz \cong M_\As|_\pzz$.

In Section \ref{subsec:RestrDefs} we will give all the definitions and in Section \ref{subsec:RestrProof} we will prove the Theorem.

\subsection{Definitions}\label{subsec:RestrDefs}

Difference equations with support on $p+\Z$ are trivial in the sense that for a $\ddA$-module $M$ which is torsion over $\C[z]$, $\tau$ induces isomorphisms $M_z\to M_{z+1}$. Hence describing these modules amounts to describing $\C[z]$-modules supported at a point. Taking limits, one can argue that difference modules over a formal neighborhood of the orbit $p+\Z$ should be just modules over $\C[[z-p]]$. 
Our definition of difference modules will contain slightly more information than just a module over $\C[[z-p]]$. Throughout, we will consider $p\in \A^1$ to be fixed and we will abbreviate $\pi=z-p$. For a $\C[z]$-module, we will denote $M_p=\C[[\pi]]\otimes_{\C[z]} M$.

\begin{definition}
The \textbf{category of difference modules on a formal disk} $\Mod^{\Z}(\pz)$ is a full subcategory of the category of diagrams $M^l\overset{i^l}\longrightarrow M\overset{i^r}\longleftarrow M^r$ where $M^l$, $M$ and $M^r$ are $\C[[\pi]]$-modules. It is defined as follows:
\begin{itemize}
\item The objects of $\Mod^{\Z}(\pz)$ are diagrams $(M^l \overset{i^l}\to M\overset{i^r}\gets M^r)\in \Mod(\C[[\pi]])$, such that
\begin{enumerate}
\item $M^l$ and $M^r$ are finite rank free $\C[[\pi]]$-modules.
\item $i^l$ and $i^r$ are injections.
\item $M/M^l$ and $M/M^r$ are torsion modules.
\end{enumerate}
\item Morphisms in $\Mod^{\Z}(\pz)$ are morphisms of diagrams of $\C[[\pi]]$-modules. In other words, morphisms $(M^l\to M\gets M^r)\to (N^l\to N \gets N^r)$ are $\C[[\pi]]$-homomorphisms $\phi:M\to N$ such that $\phi(M^l)\subseteq N^l$ and $\phi(M^r)\subseteq N^r$.
\end{itemize}
\end{definition}

We will often omit $i^l$ and $i^r$ when describing an object of $\Mod^{\Z}(\pz)$ and just write $M=(M^l,M,M^r)$, or even omit the reference to $M^l$ and $M^r$ altogether. To avoid repetition, we will use the index $lr$ to mean either $l$ or $r$.

\begin{obs}
The category $\Mod^{\Z}(\pz)$ is additive, but it is not abelian, as the cokernel of a morphism could be an object $(M^l,M,M^r)$ where $M^l$ and/or $M^r$ have torsion. It is, however, an exact category, since it is a subcategory of the (abelian) category of triples of $\C[[\pi]]$-modules, and it is closed by extensions.

For a morphism $\phi$ to be decomposed as an admissible epimorphism followed by an admissible monomorphism, it is necessary and sufficient for $\phi^l$ and $\phi^r$ to have constant rank. As we will prove later, this is the case for any restriction from a morphism in $\Hol(\ddA)$.
\end{obs}

\begin{definition}
We define the following full subcategories of $\Mod^\Z(\pz)$:
\begin{itemize}
\item $\Loc$ is the category of $(M^l,M,M^r)$ such that $M$ is finitely generated.
\item $\Locc$ is the category of $(M^l,M,M^r)$ such that $M$ is a finite dimensional $\C((\pi))$-vector space.
\end{itemize}
\end{definition}


We have all four relevant categories of difference modules in all four relevant spaces. We will now define the restriction functors.

\begin{definition}
We define the functor of \textbf{restriction to the punctured disk} $|_\pzz:\Loc\to \Locc$ as follows:
\[
M= \left(M^l,M,M^r\right) \longmapsto M|_\pzz = \left(\overline{M^l},\C((\pi))\otimes_{\C[[\pi]]}M,\overline{M^r}\right)
.\]
Where $\overline{M^{lr}}$ is defined as the image of the composition $M^{lr}\overset{i^{lr}}\to M\to \C((\pi))\otimes M$ (which is isomorphic to $M^{lr}$).
\end{definition}

We are missing the restrictions from the (punctured) line to the (punctured) formal disk. We will use the following notation to talk about this restriction.


\begin{definition}\label{def:restrictionToDisk}
We define the functor of \textbf{restriction to the disk} $|_\pz:\Hol(\ddA)\to \Loc$ as follows: For $M\in \Hol(\ddA)$, we let $L_M\in \Gl(M)$ (Definition~\ref{Def:NotReallyLattices}), $n\gg 0$ and
\[M|_\pz = ( M|_\pz^l, M_p, M|_\pz^r) :=\left( (\tau^n L_M)_p , M_p , (\tau^{-n}L_M)_p\right) = \left( \tau^n (L_M)_{p-n} , M_p , \tau^{-n} (L_M)_{p+n}\right)
.\]
$M|_\pz$ is seen as a $\C[[\pi]]$-module by identifying $\pi = z-p$. The restriction $|_\pzz:\Hol(\ddAp)\to \Locc$ is defined in exactly the same way.
\end{definition}


\begin{proposition}\label{prop:rhoWellDefdExact}
The functor $|_\pz$ has the following properties:
\begin{enumerate}
\item Its definition has no ambiguity, i.e. $|_\pz$ does not depend on $L_M\in \Gl(M)$ or a big enough $n$ (depending on $L_M$).
\item For $M\in \Hol(\ddA)$, $M|_\pz\in \Loc$, i.e. it has the following properties:
\begin{enumerate}
\item $M_p$ is a finitely generated $\C[[\pi]]$-module.
\item $M|_\pz^l$ and $M|_\pz^r$ are finite rank free $\C[[\pi]]$-modules.
\item The maps $i^{lr}:M|_\pz^{lr}\to M_p$ are injections.
\item $M_p/M|_\pz^l$ and $M_p/M|_\pz^r$ are torsion modules.
\end{enumerate}
\item The functor $|_\pz$ maps morphisms in $\Hol(\ddA)$ to morphisms in $\Loc$.
\item It is an exact functor, in the sense of exact categories: it maps short exact sequences to short exact sequences.
\item Let $f:M\to N$ be a morphism in $\Hol(\ddA)$. Then $f|_\pz^{lr} :M|_\pz^{lr}\to N|_\pz^{lr}$ is a morphism of free $\C[[\pi]]$-modules that has constant rank, i.e. its cokernel is torsion-free.
\end{enumerate}
The functor $|_\pzz$ has the exact same properties, except for 2a): if $M\in \Hol(\ddAp)$, $M_p$ is a finite dimensional $\C((\pi))$-vector space.
\end{proposition}

%
%

\begin{proof}
\leavevmode
\begin{enumerate}
\item Consider the zeroes and poles of $L_M$: $Z_{L_M}$ and $P_{L_M}$ respectively. 
Note that if $p+m\notin P_{L_M}$, the isomorphism $\tau:M_{p+m}\to M_{p+m+1}$ maps $(L_M)_{p+m}$ inside of $(L_M)_{p+m+1}$, and if $p+m\notin Z_{L_M}$, the analogous statement is true for $\tau^{-1}:M_{p+m+1}\to M_{p+m}$. Therefore, if $n$ is big enough (for example, if $Z_{L_M}\cup P_{L_M}\subseteq [p-n,p+n-1]$), $\tau$ identifies all the modules $(L_M)_{p+n+m}\subseteq M_{p+n+m}$ for $m\ge 0$, and similarly it identifies all the modules $(L_M)_{p+n-m}\subseteq M_{p-n-m}$.

Suppose now a different $L'_M\in \Gl(M)$ is chosen. Without loss of generality, taking the intersections allows us to assume that $L'_M\subseteq L_M$. The module $L_M/L'_M$ is finite dimensional. Therefore, for big enough $n$, $(L'_M)_{p\pm n}= (L_M)_{p\pm n}$, as submodules of $M_{p\pm n}$. This shows that $M_\pz$ doesn't depend on the choices of $L\in \Gl(M)$ or a big enough $n$, as desired.

\item 

\begin{enumerate}
\item This follows from Proposition \ref{prop:finiteStalks}.
\item We must prove that for some $L_M\in \Gl(M)$, $(L_M)_{p\pm m}$ is a finite rank free $\C[[\pi]]$-module, where $m\gg 0$. This follows from the fact that $L_M$ is a finitely generated $\C[z]$-module, and therefore its torsion is a finite length module. This implies that if $n\gg 0$, the support of the torsion of $L_M$ will be contained in $[p-n,p+n]$, $(L_M)_{p\pm m}$ will be torsion free if $m\ge n$.

\item Since $\C[[\pi]]$ is a flat $\C[z]$-module, the maps $M|_\pz^l\to M_p\gets M|_\pz^r$ are automatically injections.
\item This follows from the fact that $M/L_M$ is itself a torsion module.
\end{enumerate}

\item Consider a $\ddA$-homomorphism $f:M\to M'$, and let us show that the corresponding map $f|_\pz $ is a morphism in $\Loc$. Pick some $L_M\in \Gl( M)$, and $L_{M'}\in \Gl(M')$ containing $f(L_M)$. Then, for big enough $n$, $\tau^{\pm n}(L_M)_{p\mp n}=M|_\pz^{lr}$ and $\tau^{\pm n}(L_{M'})_{p\mp n}=M'|_\pz^{lr}$, so there are indeed maps $ f|_\pz^l:\tau^{n}(L_M)_{p- n}\to \tau^{n}(L_{M'})_{p-n}$ and $f|_\pz^r:\tau^{-n}(L_M)_{p+ n}\to\tau^{-n} (L_{M'})_{p+n}$ that are restrictions of $f_p:M|_\pz\to M'|_\pz$. These maps are uniquely determined, since the maps $i'^{lr}:M'|_\pz^{lr}\to M'|_\pz$ are injective by the previous part.

\item Consider a short exact sequence $0\to A\overset{f}\to B\overset{g}\to C\to 0$ in $\Hol(\ddA)$, and let $L_B\in \Gl(B)$. Let $L_{A}=L_B\cap A$. It's a submodule of $L_B$, and therefore it's a finitely generated $\C[z]$-module, and $A/L_{A}$ embeds into $B/L_B$, so it is a torsion module. Further, choose $L_C\in \Gl(C)$ such that it contains $g(L_B)$. Then, $L_{A}\in \Gl(A)$, and now we observe that $A|_\pz^{lr}=\tau^{\pm n}(L_A)_{p\mp n}=\ker(
\tau^{\pm n}(L_B)_{p\mp n} \longrightarrow \tau^{\pm n}(L_{C})_{p\pm n}
)= \ker(
B|_\pz^{lr} \longrightarrow C|_\pz^{lr}
)$. It follows that $ A|_\pz=\ker(g|_\pz)$.

For the exactness on the right, we observe that $g(L_B)\in \Gl(C)$. Since $g(L_B)\cong L_B/f(L_A)$, we have that $\coker(f|_\pz^{lr} )=C|_\pz^{lr}$.

\item This follows by decomposing every morphism in $\ddA$ into a surjection followed by an injection, and then applying the exactness of $|_\pz$.

\end{enumerate}

The proof for $|_\pzz$ is analogous.

\end{proof}

The following statement is straightforward.

\begin{proposition}\label{prop:squareCommutes}
The diagram (\ref{eq:Fibersquare}) is commutative, in the sense that there's a natural isomorphism $|_\pzz \circ |_\pz\cong |_\pzz \circ |_\As$.
\end{proposition}

We can now prove Theorem~\ref{thm:formalCycles}. The proof can be found in Section~\ref{subsec:RestrProof}. The reader can now skip to it, since it does not use any statements other than the ones proved up to this point.

\begin{proposition}\label{prop:stalksInsteadofFibers}
Theorem \ref{thm:formalCycles} works the same when ``formal fibers'' are replaced by ``stalks''. Let $V_p=\Spec\C[z]_{(z-p)}$, and let $\Hol(V_p)$, $\Hol(V_p^*)$, $|_{V_p}$ and $|_{V_p^*}$ be defined as $\Loc$, $\Locc$, $|_\pz$ and $|_\pzz$, replacing $\C[[\pi]]$ everywhere by $\C[z]_{(z-p)}$ and $\C((\pi))$ by $\C(z)$. The following diagram is a fiber product of categories:
\[
\begin{tikzcd}[column sep = 5 em,ampersand replacement=\&]
\Hol(\ddA) \arrow[r,"|_{\As}"]\arrow[d,"|_{V_p}"] \& \Hol(\ddAp)\arrow[d,"|_{V_p^*}"] \\
\Hol(V_p) \arrow[r,"|_{V_p^*}"] \& \Hol(V_p^*).
\end{tikzcd}
\]
\end{proposition}
\begin{proof}
Consider the diagram
\[
\begin{tikzcd}[column sep = 5 em,ampersand replacement=\&]
\Hol(\ddA) \arrow[d,"|_{\As}"]\arrow[r,"|_{V_p}"] \& \Hol(V_p)\arrow[d,"|_{V_p^*}"]\arrow[r,"\C{\left[\left[\pi\right]\right]}\otimes_{\C\left[z\right]}"] \& \Loc \arrow[d,"|_{\pzz}"]\\
\Hol(\ddAp) \arrow[r,"|_{V_p^*}"] \& \Hol(V_p^*)\arrow[r,"\C{\left[\left[\pi\right]\right]}\otimes_{\C\left[z\right]}"] \& \Locc.\\
\end{tikzcd}
\]
\vspace{-2.5em}

The horizontal arrows on the right are given by $(M^l,M,M^r)\mapsto (\C[[\pi]]\otimes M^l,\C[[\pi]]\otimes M,\C[[\pi]]\otimes M^r)$. We claim that it is a fiber square, from which it follows that Theorem \ref{thm:formalCycles} implies the statement by the following abstract fact: if the outer rectangle is cartesian and the right square is cartesian as well, then the left square is cartesian.

First of all, we check that $\Mod(\C[z]_{(z)}) \cong \Mod(\C[[z]])\times_{\Mod(\C((z)))} \Mod(\C(z))$. The functor $F$ from left to right is given by tensoring. The inverse functor $G$ is given by mapping a triple $(M_{\C[[z]]},M_{\C(z)},\phi:\C((z))\otimes M_{\C[[z]]} \cong \C((z))\otimes M_{\C(z)})$ to the kernel of the map $M_{\C[[z]]}\oplus M_{\C(z)}\to \C((z))\otimes M_{\C[[z]]} \cong \C((z))\otimes M_{\C(z)}$. The fact that $GF\cong \Id$ comes from tensoring with the short exact sequence $0\to\C[z]_{(z)}\to \C[[z]]\oplus \C(z) \to \C((z))\to 0$. To see that $FG\cong \Id$, notice that localizing the short exact sequence $S = (0\to GM \to M_{\C[[z]]}\oplus M_{\C(z)} \to \C(z)\otimes M_{\C[[z]]}\to 0)$ yields $\C(z)\otimes GM \cong M_{\C(z)}$, and we can apply the five lemma to the natural map of short exact sequences from $0\to GM\to \C[[z]]\otimes GM\oplus \C(z)\otimes GM\to \C((z))\otimes GM\to 0$ to $S$. This yields the isomorphism $FG \cong \Id$. Note that the naturality of the isomorphisms implies that $F$ and $G$ are inverse when acting on morphisms as well.

%
%
%
%

Given that $\Mod(\C[z]_{(z-p)}) \cong \Mod(\C[[\pi]])\times_{\Mod(\C((\pi)))} \Mod(\C(z))$, we can now prove that $\Hol(V_p)\cong \Hol(U_p)\times_{\Hol(U_p^*)} \Hol(V_p^*)$. Start by noting that the functor $(M^l,M,M^r)\mapsto M$ induces an equivalence $\Hol(U_p)\times_{\Hol(U_p^*)} \Hol(V_p^*)\cong \Hol(U_p)\times_{\Mod(\C((\pi)))}\Mod(\C(z))$: the inverse of this functor is given by $(M_1,M_2)\mapsto (M_1,M_2)$, with $M_2^{lr}=G(M_1^{lr},M_2)$. Now notice that $ \Hol(U_p)\times_{\Mod(\C((\pi)))}\Mod(\C(z))$ is the category of diagrams $a^l\to a\gets a^r$ with objects in $\Mod(\C[[\pi]])\times_{\Mod(\C((\pi)))} \Mod(\C(z))$ such that after applying $G$ they land in $\Hol(\pz)$: this is induced by the functor
\[
\left( ((M_1^l\to M_1\gets M_1^r),M_2) \right)\longleftrightarrow \left((M_1^l,M_2) \to(M_1,M_2) \gets (M_1^r,M_2)  \right).
\]
Now the claim follows from the fact that $F$ and $G$ are mutual inverses.

\end{proof}

\subsection{Extending a difference module over a puncture}\label{sec:extendOverPuncture}

One application of Theorem \ref{thm:formalCycles} is to compute all the possible ways that a difference module on the punctured line can be extended to a module on the whole affine line.

\begin{remark}\label{rem:GMextension}
Let $M\in \Hol(\ddAp)$. $j_{!*}M$ is the module given by the following information: $(j_{!*}M)|_\As=M$ and
\[
(j_{!*}M)|_\pz = \left(
M|_\pzz^l,M|_\pzz^l+M|_\pzz^r,M|_\pzz^r 
\right) \subset M|_\pzz = \left(
M|_\pzz^l,M_p,M|_\pzz^r 
\right).
\]
If a module $N$ is supported on $p+\Z$, then $N|_\pz^{lr}=0$ and $N_p$ is torsion. We can see directly that our candidate for $j_{!*}M$ has no such (nonzero) submodules or quotients, because $(j_{!*}M)|_\pz$ doesn't have them. Proposition~\ref{Prop:GMExtensionProps}, ensures that this is indeed the intermediate extension.
\end{remark}

\begin{remark}~\label{rem:apparent}
Given a $N\in \Hol(\ddA)$ coming from a difference equation via Construction~\ref{con:diffEq}, we can think of the difference $\frac{N}{j_{!*}(N|_\As)}$ as \emph{apparent} singularities, because they disappear after the change of coordinates that transforms $N$ into $j_{!*}(N|_\As)$. The remark above shows that the local type can detect these, since they correspond to the module $\frac{N_p}{N|_\pz^l+N|_\pz^r}$.
\end{remark}

Now let $N$ be any other module in $\Hol(\ddA)$ such that $N|_\As=M$. The adjunction between $|_{\ddAp}$ and the forgetful functor yields a natural map $N\to M$, whose image we will denote $\ov N$, and Proposition~\ref{Prop:GMExtensionProps} ensures that $j_{!*}M\subseteq \ov N \subseteq M$. Therefore, we have a diagram
\[
N \twoheadrightarrow \ov N \hookleftarrow j_{!*}M.
\]
The kernel of the first arrow and the cokernel of the second are torsion modules supported on $p+\Z$, so they are successive extensions of $\delta=\ddA/ \ddA (z-p)$. Therefore, to understand the collection of possible $N$'s it is sufficient to understand extensions of modules by torsion modules. The following proposition computes all extensions in ${\Mod^\Z(\pz)}$. Notice that even though ${\Mod^\Z(\pz)}$ is just an exact category, since it is closed under extensions in the category of diagrams of $\C[[\pi]]$-modules, $\Ext$ groups can be computed in this larger abelian category, and the same is true for $\Loc$ and $\Locc$.

\begin{proposition}\label{prop:computingExtensions}
Let $M=(M^l,M,M^r)$ and $N=(N^l,N,N^r)$ be two modules in ${\Mod^\Z(\pz)}$. Let us denote $\Phi^{lr}N = \frac{N^{lr}}{N}$. There is an exact sequence:
\[
\begin{array}{rrrl}
 0   \longrightarrow &
 \Hom_{\Mod^\Z(\pz)}(M,N)   \overset{\Theta}\longrightarrow &
 \Hom_{\C[[\pi]]}(M,N)
\longrightarrow &
\\
\longrightarrow  \Hom_{\C[[\pi]]}(M^l,\Phi^lN)\oplus \Hom_{\C[[\pi]]}(M^r,\Phi^r N) \longrightarrow &
 \Ext^1_{\Mod^\Z(\pz)}(M,N) \overset{\Theta}\longrightarrow &
  \Ext^1_{\C[[\pi]]}(M,N)
\longrightarrow  &
 0.
\end{array}
\]
Where $\Theta$ is the forgetful functor from ${\Mod^\Z(\pz)}$ to $\Mod(\C[[\pi]])$, and $\Hom_{\C[[\pi]]}(M,N)$ is mapped into the group $\Hom_{\C[[\pi]]}(M^l,\Phi^l N)\oplus \Hom_{\C[[\pi]]}(M^r,\Phi^r N)$ by restricting to $M^{lr}$ and composing with the projection to $N/N^{lr}$.
\end{proposition}

\begin{corollary}\label{cor:extensions}
\begin{enumerate}
\item If $M$ is torsion, then
\[
\Ext^1_{\Mod^\Z(\pz)}(M,N) \cong \Ext^1_{\C[[\pi]]}(M,N).\]
\item  If $M$ is torsion-free, then
\[
\Ext^1_{\Mod^\Z(\pz)}(M,N) \cong \frac{\Hom_{\C[[\pi]]}(M^l,\Phi^l N)\oplus \Hom_{\C[[\pi]]}(M^r,\Phi^r N)}{\Hom_{\C[[\pi]]}(M,N) / \Hom_{{\Mod^\Z(\pz)}}(M,N)}.
\]
\end{enumerate}
\end{corollary}

This corollary is enough to compute all the possible extensions of a module $M\in\Hol(\ddAp)$ to some $N\in\Hol(\ddA)$. Going back to the diagram $N \twoheadrightarrow \ov N \hookleftarrow j_{!*}M
$, the first part can be applied to compute the possible $\ov N$'s from $j_{!*}M$, which amounts to taking a finitely generated submodule $\ov N$ such that $(j_{!*}M)_p\subseteq \ov N \subseteq M_p$. The second part of the corollary can then be applied to obtain all possible extensions of $\ov N$ by a torsion module.

\begin{proof}[Proof of Proposition \ref{prop:computingExtensions}]
The functor $\Theta$ produces a homomorphism $\Ext^1_{\Mod^\Z(\pz)}(M,N) \overset{\Theta}\longrightarrow \Ext^1_{\C[[\pi]]}(M,N)$. Let us show that it is surjective. Consider an extension of $\C[[\pi]]$-modules $0\to N\overset{f}\to P\overset{g}\to M\to 0$. We need to find two submodules $P^{lr}\subset P$ such that there are induced short exact sequences $0\to N^{lr}\to P^{lr}\to M^{lr}\to 0$. Since $M^{lr}$ is a projective $\C[[\pi]]$-module, there is a lift $\wt i^{lr}:M^{lr}\to P$ such that $g\circ \wt i^{lr}=i^{lr}:M^{lr}\to M$. Then, one can take $P^{lr} = fN^{lr}+\wt i^{lr} M^{lr}$. Since $fN^{lr}\subset \ker g$ and $g|_{\wt i^{lr} M^{lr}}$ is injective, this sum is actually a direct sum, so we indeed obtain the desired short exact sequences, and thus some $P = (P^l,P,P^r)\in {\Mod^\Z(\pz)}$ fitting into a short exact sequence $0\to N\to P\to M\to 0$.

The kernel of $\Theta$ is the group of extensions that take the form
\[
0\longrightarrow N \longrightarrow (P^l,N\oplus M,P^r)\longrightarrow M\longrightarrow 0.
\]
These extensions are all given by choosing submodules $P^{lr}\subset N\oplus M$, such that they contain $N^{lr}$ and the quotients map isomorphically into $M^{lr}$. The short exact sequence $0\to N^{lr}\to P^{lr} \to M^{lr}\to 0$ splits because $M^{lr}$ is projective. Let $\wt i^{lr}:M^{lr}\to P^{lr}\to N\oplus M$ be one such splitting. The $M$ component of $\wt i^{lr}$ must be the identity, and therefore, it is determined by a map $M^{lr}\to N$. Therefore we have a map $\Hom_{\C[[\pi]]}(M^l,N)\oplus \Hom_{\C[[\pi]]}(M^l,N)\to \Ext^1_{\Mod^\Z(\pz)}(M,N)$ whose image is the kernel of $\Theta$.

Let us show that this map is a $\C[[\pi]]$-module homomorphism: multiplication by $\C[[\pi]]$ is induced on both sides by multiplication on $N$, so it is clear that it commutes with $\C[[\pi]]$. For the sum, we can use the Baer sum. For two maps $(j^l,j^r):M^l\oplus M^r\to N$, the corresponding extension is the class of the following module, with the obvious structure of an extension of $M$ by $N$:
\[
P = M\oplus N; P^{lr} =  (\Id_M,j^{lr}) M^{lr} + N^{lr}.
\]
Suppose we have two pairs of maps, $j^{lr}_i$, for $i=1,2$, giving rise to two extensions $P_i$. Their Baer sum is by definition:
\[
P_3 = \frac{P_1\times_M P_2}{\{(a,0)-(0,a):a\in N \}}\cong \frac{N\oplus N\oplus M}{N} 
.\]
In the last term, $N$ is embedded in $N\oplus N$ diagonally, so $P_3$ is isomorphic to $N\oplus M$, via the map $(n_1,n_2,m)\mapsto (n_1+n_2,m)$. Looking at this map, we see that the image of $M^{lr}$ by $j^{lr}$ in $P_3$ is given by $(j_1^{lr}+j_2^{lr},\Id_M)$, as desired.

We have an exact sequence
\[
\Hom_{\C[[\pi]]}(M^l,N)\oplus \Hom_{\C[[\pi]]}(M^r,N)\longrightarrow \Ext^1_{\Mod^\Z(\pz)}(M,N) \overset{\Theta}\longrightarrow \Ext^1_{\C[[\pi]]}(M,N)
\longrightarrow 0.
\]
Let us compute the kernel of the leftmost map. It is made of the pairs of maps $(j^l,j^r)\in\Hom_{\C[[\pi]]}(M^l,N)\oplus \Hom_{\C[[\pi]]}(M^r,N) $ such that the following short exact sequence is split:
\[
0\longrightarrow N \longrightarrow (N^l\oplus ( j^l,1) M^l,N\oplus M,N^r\oplus ( j^r,1) M^r)\overset{p}\longrightarrow M\longrightarrow 0
.\]
These short exact sequences split exactly when there is a section of the second arrow, i.e. a map $ s:M\to N\oplus M$ such that $p\circ s=1_{M}$. This means that $s$ is of the form $s=(j,1)$. Further, in order to be a morphism, $s$ must map $M^{lr}$ inside of $(N\oplus M)^{lr} = N^{lr}\oplus ( j^{lr},i^{lr}) M^{lr}$. Since $s^{lr} = (j^{lr},1)$ does map $M^{lr}$ into $(N\oplus M)^{lr}$, we have that
\[
s(M^{lr})\subseteq (N\oplus M)^{lr} \Leftrightarrow (s|_{M^{lr}}-s^{lr}) (M^{lr}) \subseteq (N\oplus M)^{lr}.
\]
Now, $(s|_{M^{lr}}-s^{lr}) = (j-j^{lr}, 0)$, so $s$ is a morphism if and only if $j-j^{lr}$ maps $M^{lr}$ into $N^{lr} = (N\oplus M)^{lr}\cap N\oplus 0$. Therefore, the above extension is trivial if and only if there exist three maps as follows:
\[
(a^l,j,a^r)\in \Hom_{\C[[\pi]]}(M^l, N^l) \oplus \Hom_{\C[[\pi]]}(M, N) \oplus \Hom_{\C[[\pi]]}(M^r, N^r).
\]
This triple must have the property that $j^{lr} = j + a^{lr}$. Using the fact that $M^{lr}$ is free, we can simplify the quotient:
\[
\frac{\frac{\Hom_{\C[[\pi]]}(M^l, N)}{\Hom_{\C[[\pi]]}(M^l, N^l)}\oplus \frac{\Hom_{\C[[\pi]]}(M^r, N)}{\Hom_{\C[[\pi]]}(M^r, N^r)}}{\im(\Hom_{\C[[\pi]]}(M, N))} = 
\frac{\Hom_{\C[[\pi]]}(M^l, \Phi^l N)\oplus \Hom_{\C[[\pi]]}(M^r, \Phi^r N)}{\im(\Hom_{\C[[\pi]]}(M, N))}.
\]
Finally, note that the maps in $\Hom_{\C[[\pi]]}(M, N)$ the become $0$ in $\Hom(M^l, \Phi^lN)\oplus \Hom(M^r, \Phi^rN)$ are exactly those in $\Hom_{\Mod^\Z(\pz)}(M,N)$.
\end{proof}

\section{Vanishing cycles and singularities of difference equations}

In this section we define vanishing cycles and we show some properties that suggest it is a good analogue for the functor of vanishing cycles in the case of $D$-modules. As always, we will fix an orbit $p+\Z\subset \A^1$. We start by recalling Definition~\ref{def:vanishingCycles}.

\begin{definition}
We define the \textbf{functor of (left) vanishing cycles} $\Phi^l_{p+\Z}:\Hol(\ddA)\to \Mod(\C[[\pi]])$ by
\[
\Phi^l_{p+\Z}(M) = \frac{M_p}{M|_\pz^l}.
\]
Which can be made into a functor in the obvious way.
\end{definition}

Throughout this section we may abbreviate $\Phi= \Phi^l_{p+\Z}$.

\begin{remark} We can make the following observations.
\begin{enumerate}
\item By Proposition~\ref{prop:rhoWellDefdExact}, $\Phi^l_{p+\Z}M$ is a finite length $\C[[\pi]]$-module.
\item The analogous definition yields a second functor $\Phi^r_{p+\Z}(M) = \frac{M_p}{M|_\pz^r}$. Every statement in this section has an analogous statement for $\Phi^r_{p+\Z}$ after interchanging the roles of $\tau$ and $\tau^{-1}$.
\item $\Phi^l_{p+\Z}$ is exact: it is a composition of two exact functors between exact categories $\Hol(\ddA)\to \Hol(\pz)\to \Mod(\C[[\pi]])$. Since its source and target are abelian, it is indeed an exact functor in the sense of abelian categories.
\end{enumerate}
\end{remark}

One reason why $\Phi^l_{p+\Z}$ is a good replacement for vanishing cycles is that it vanishes exactly for modules with no zeroes (and $\Phi^r_{p+\Z}$ vanishes for modules with no poles). We also show exactly how to compute the local type and vanishing cycles from a matrix difference equation.

\subsection{Relation to singularities of difference equations}\label{sec:whenIsModuleSingular}


We show that many reasonable notions of zeroes and poles are equivalent. In particular, we can describe when a difference module $M\in \Hol(\ddA)$ has a ``zero'' or a ``pole'' in terms of the underlying $\C[z]$-module.


\begin{proposition}\label{prop:WhatisaZero}
Let $M\in \Hol(\ddA)$. The following are equivalent:
\begin{enumerate}
\item $\Phi^l_{p+\Z} M = 0$ for every $p\in \A^1$ (resp. $\Phi^r_{p+\Z} M =0$).
\item For some $L\in \Gl(M)$, there's some $N\in \mathbb R$ such that $M/L$ is supported on $\{a\in \C:\Re(a)\ge N\}$ (resp. $\le N$).
\item For every $L\in \Gl(M)$, there's some $N\in \mathbb R$ such that $M/L$ is supported on $\{a\in \C:\Re(a)\ge N\}$ (resp. $\le N$).
\item Any finite subset of $M$ is contained in some $L\in \Gl(M)$ with no zeroes (resp. no poles), i.e. such that $\tau^{-1}L\subseteq L$ (resp. $L\subseteq \tau L$).
\item $M$ is finitely generated over $\C[z]\langle\tau\rangle$ (resp. $\C[z]\langle\tau^{-1}\rangle$).
\end{enumerate}
\end{proposition}

\begin{proof}\leavevmode
\begin{itemize}
\item[$1\Rightarrow2$] Suppose that $\Phi^l M=0$, and let $L\in \Gl(M)$. The fact that all the vanishing cycles are $0$ implies that for every $p\in \A^1/\Z$ and $n\gg 0$, $M_p = \tau^n(L_{p-n})$. Since $\tau^n$ induces an isomorphism $M_{p-n}\cong M_p$, we deduce that $M_{p-n}=L_{p-n}$ for some big enough $n$ (which can be chosen uniformly for all $p$'s, since only a finite set affect the count by Lemma \ref{lem:glattices}). The conclusion follows.

\item[$2\Rightarrow 3$] Two modules $L_1,L_2\in \Gl(M)$ can only differ at a finite set, since $L_i/(L_1\cap L_2)$ are finite length modules.
\item[$3\Rightarrow 4$] Choose any $L\in \Gl(M)$ containing a given finite set. Note that $\supp \frac{\tau^{-n}L+\tau^{-n+1}L}{\tau^{-n+1}L} = Z_L-n+1$. Therefore, if $n$ is big enough, we take $L'=\tau^{-n}L+\cdots +L\in \Gl(M)$, and we have that
\[Z_{L'}=
\supp\frac{L'+\tau L'}{\tau L'}\subseteq \supp  \frac{\tau^{-n}L+\tau^{-n+1}L}{\tau^{-n+1}L} \cap \supp \frac{M}{L} =(Z_L-n+1)\cap \{z:\Re z\gg 0 \}= \emptyset
.\]
\item[$4\Rightarrow 5$] Let $L\in \Gl(M)$ have no zeroes, chosen to contain a (finite) generating set of $M$ over $\ddA$. Then, $\tau^{-1}L\subseteq L$, which implies that $M=\sum_{n\in \Z} \tau^nL = \sum_{n\ge 0} \tau^nL$, so a finite set generating $L$ over $\C[z]$ also generates $M$ over $\C[z]\langle \tau\rangle$, as desired.
\item[$5\Rightarrow 1$] Suppose that $M$ is finitely generated as a $\C[z]\langle\tau\rangle$-module, and let $S$ be a finite generating set over $\C[z]\langle \tau\rangle$. Let $L\in \Gl(M)$ containing $S$. By assumption, $\sum_{i\ge 0} \tau^iL=M$. In particular, $\tau^{-1}L\subset L +\tau L+\cdots+\tau^mL$ for some $m$, since $\tau^{-1}L=\C[z](\tau^{-1}S)$, so it is finitely generated. Let $L'=\tau+\cdots+\tau^mL$. Then $\tau^{-1}L'\subseteq L'$, so the sequence $\tau^i L'$ is increasing with $i$, and we have that $\sum_{i\ge 0} \tau^iL'=\bigcup_{i\ge 0} \tau^iL'=M$. Fixing a fiber $p$, we have that $\bigcup_{i\ge 0} \tau^i(L_{p-i}') =\bigcup_{i\ge 0} (\tau^iL')_p =M_p$. We have that $M_p/L_p$ is a finite length module, by definition of $\Gl(M)$ and Proposition \ref{prop:finiteStalks}. Therefore there is an $N\gg 0$ such that $\tau^n(L'_{p-n})=M_p$ for any $n\ge N$. By definition of $|_\pz$, this implies that $M_p=M|_\pz^l$, so $\Phi^l_{p+\Z} M=0$.
\end{itemize}
\end{proof}

\begin{remark}
In particular, note that if $M$ comes from a matrix difference equation $y(z+1)=A(z)y(z)$ (Construction~\ref{con:diffEq}), then if $A(z)$ is defined everywhere on $p+\Z$, $\Phi^r_{p+\Z}M=0$. Conversely, if $\Phi^r_{p+\Z}M=0$, then there is a gauge change after which $A$ becomes defined everywhere on $p+\Z$. The same holds for $A^{-1}$ and $\Phi^l_{p+\Z}$.
\end{remark}

These can be put together to characterize difference modules with no singularities at finite points.

\begin{corollary}
Let $M\in \Hol(\ddA)$. The following are equivalent:
\begin{enumerate}
\item $\Phi^l_{p+\Z} M = \Phi^r_{p+\Z} M = 0$ for every $p\in \A^1$.
\item $M$ is finitely generated over $\C[z]$.
\item $M$ is a vector bundle.
\end{enumerate}
\end{corollary}
\begin{proof}
\begin{itemize}
\item[1$\Rightarrow$2] Let $L\in \Gl(M)$. By Proposition \ref{prop:WhatisaZero}, $M/L$ is supported on a set of the form $\{a\in \C:-N\le \Re(a)\le N \}$ for some $N\in \mathbb R$, which combined with Lemma \ref{lem:glattices} implies that $M/L$ has finite support. Finally, Proposition \ref{prop:finiteStalks} implies that $M/L$ is finitely generated, so $M$ is indeed finitely generated over $\C[z]$.


\item[2$\Rightarrow$3] If $M$ is finitely generated over $\C[z]$ and it is not a vector bundle, it must have a torsion element. Let $s\in M\setminus \{0\}$ be such that its support is a single point $a$. Then $\supp \tau^n s = a+n$, which implies that the torsion submodule of $M$ is not finitely generated, and therefore $M$ itself is not finitely generated.
\item[3$\Rightarrow$1] This follows directly from Proposition \ref{prop:WhatisaZero}.

\end{itemize}
\end{proof}

Consider a matrix difference equation $y(z+1)=A(z)y(z)$ and use Construction~\ref{con:diffEq} to construct a difference module $M$ with a trivial bundle $L\subseteq M$. We will now discuss how the local type of $M$ can be computed directly from the matrix. The answer is most convenient when all the zeroes of $A$ are to the left of its poles, which is the opposite situation to the ``austerity'' used to construct the intermediate extension.

\begin{proposition}\label{prop:matrices}
Consider a matrix difference equation $y(z+1)=A(z)y(z)$ with $A\in GL_n(\C(z))$, and consider the corresponding difference module $\C[z]^n=L\subseteq M\subseteq \C(z)^n$. Let $P_A\subset \C$ be the subset of $p+\Z$ over which $A(z)$ is not defined, and let $Z_A$ be the subset of $p+\Z$ over which $A(z)^{-1}$ is not defined. %
Consider the obvious ordering on $p+\Z$. We define the following sequence of matrices $(A^{(n)}(z))_{n\in \Z}$:
\[
A^{(0)}(z)=\Id\in \GL_n(\C(z));\quad A^{(n+1)}(z)= A^{(n)}(z) A(z+n)^{-1};\quad A^{(n-1)}(z) = A^{(n)}(z)A(z+n-1).
\]
It is straightforward to check that $\tau^{-n}L$ is generated over $\C[z]$ by the columns of $A^{(n)}$.
\begin{enumerate}
\item Let $n_1,n_2\in \Z$ be such that $p+n_1\le Z_A\cup P_A<p+n_2$. Then:
\[
M|_\pz^l = (\tau^{-n_1}M)_p;\quad 
M|_\pz^r = (\tau^{-n_2}M)_p.
\]
In terms of matrices, there exists a basis for $M|_\pz^l$ such that the coordinates for a basis of $M|_\pz^r$ are given by the columns of $(A^{(n_1)})^{(-1)}A^{(n_2)}$. Since a basis for $M|_\pz^l$ is a basis for $\C((z-p))\otimes M_p$, this computes $M|_\pzz$.

\item Additionally, Let $n_3\le n_4$ be such that $p+n_3\le P_A$ and $Z_A < p+n_4$. In this case, $M_p$ is generated by $\{(\tau^{-n}L)_p \}_{n_3\le n\le n_4}$. In the standard basis of $L$, $M_p$ is generated by the columns of $\{A^{(n)}\}_{n_3\le n\le n_4}$. In this case, $M|_\pz^l=(\tau^{-n_1}L)_p$ and $M|_\pz^r = (\tau^{-n_2}L)_p$, so in the standard basis $M|_\pz^l$ (resp. $M|_\pz^r$) is the column span of $A^{(n_1)}$ (resp. $A^{(n_2)}$).

\item In particular, suppose that $n_3,n_4$ above can be chosen so that $n_3=n_4$. Then $M_p=(\tau^{-n_3}M)_p$. In terms of matrices, there is a basis of $M_p$ such that $M|_\pz^l$ (resp. $M|_\pz^r$) is the column span of $(A^{(n_3)})^{-1}A^{(n_1)}$ (resp. $(A^{(n_3)})^{-1}A^{(n_2)}$) over $\C[[z-p]]$.
\end{enumerate}

\end{proposition}
\begin{proof}
Let $n\in \Z$. Since the difference module structure is $\tau^{-1}(y(z))=A(z)^{-1}y(z+1)$, $A(p+n)^{-1}$ being well-defined is equivalent to $\tau^{-1}(L_{p+n+1})\subseteq L_{p+n}$, or equivalently, $L_{p+n+1}\subseteq \tau(L_{p+n})=(\tau L)_{p+n+1}$. Analogously, $p+n\notin P_A$ if and only if $(\tau L)_{p+n+1}=\tau(L_{p+n})\subseteq  L_{p+n+1} $. Applying $\tau^{-1- n}$, we obtain the following relations:
\[
p+n\notin Z_A\Leftrightarrow (\tau^{-n-1}L)_p\subseteq (\tau^{-n}L)_p;\quad 
p+n\notin P_A\Leftrightarrow (\tau^{-n}L)_p\subseteq (\tau^{-n-1}L)_p. 
\]
Let $C_n=(\tau^{-n}L)_p\subseteq M_p$, and consider the sequence $(C_n)$. We have just shown that $C_n\subseteq C_{n+1}$ (resp. $C_n\supseteq C_{n+1}$) whenever $p+n\notin P_A$ (resp. $p+n\notin Z_A$). Thus:
\begin{align*}
p+n_1\le Z_A \cup P_A&\Longrightarrow C_n= C_{n_1}\quad \forall n\le n_1; &
Z_A \cup P_A < p+n_2 &\Longrightarrow C_{n_2}= C_{n}\quad \forall n\ge n_2 ;\\
p+n_3\le P_A &\Longrightarrow C_{n}\subseteq  C_{n_3}\quad \forall n\le n_3; & 
Z_A< p+n_4 &\Longrightarrow C_{n}\subseteq  C_{n_4}\quad \forall n\ge n_4.
\end{align*}
Therefore, when $n\gg 0$, $M|_\pz^l=(\tau^{n}L)_p = C_{-n}=C_{n_1}$, and similarly $M|_\pz^r = C_{n_2}$. Also, $M_p$ is generated by $(\tau^{-n}L)_p$ for $n\in \Z$. For $n\le n_3$, $(\tau^{-n}L)_p\subseteq (\tau^{-n_3}L)_p$ and for $n\ge n_4$, $(\tau^{-n}L)_p\subseteq (\tau^{-n_4}L)_p$, so to generate $M_p$ only the above modules for $n_3\le n \le n_4$ are required. In particular, if $n_3=n_4$, $M_p=(\tau^{-n_3}L)_p$.

All the statements about matrices are straightforward given that the columns of $A^{(n)}$ generate $\tau^{-n}L_p$.


\end{proof}

\begin{remark}
Note that the matrices $(A^{(n)})^{-1}A^{(m)}$ have the following simple expression, which gets simpler as $|n-m|$ gets smaller:
\begin{align*}
n>m\Rightarrow (A^{(n)})^{-1}A^{(m)}&=A(z+n-1)A(z+n-2)\cdots  A(z+m);\\ n<m \Rightarrow (A^{(n)})^{-1}A^{(m)}&=A(z+n)^{-1}A(z+n+1)^{-1}\cdots A(z+m-1)^{-1}.
\end{align*}
\end{remark}

\subsection{Relation to the monodromy matrix}\label{sec:relationToBorodin}

In this section, concretely as Proposition~\ref{prop:relationToBorodin}, we describe the relation between the \textit{monodromy matrix} described in \cite{Borodin}, which comes from the results of \cite{Birkhoff}, and the local type.

In this section we will work in the analytic topology and with holomorphic functions. We will let $\M_X$ and $\cH_X$ be the sheaf of meromorphic and holomorphic functions respectively on a complex manifold $X$. We are also using $\M$ for the Mellin transform, but it will not make an appearance here, nor will we work in the analytic topology outside of this section.

We will state all the results in the section followed by the proofs.

\begin{theorem}[{\cite[Theorem III]{Birkhoff}}, {\cite[Theorem 1.3]{Borodin}}]\label{thm:Borodin}
Let $A(z)\in \GL_n(\C(z))$, and consider the difference equation $y(z+1)=A(z)y(z)$. There are matrices $Y^l,Y^r\in \GL_n(\M_\C)$ with the following properties:
\begin{itemize}
\item They solve the difference equation: $Y^{lr}(z+1)=A(z)Y^{lr}(z)$.
\item There is a left (resp. right) half-plane over which $Y^l$ (resp. $Y^r$) is holomorphic and invertible. By a left half-plane we mean a set of the form $\{z\in \C:\Re z<N\}$. 
\end{itemize}
If $Y'^{lr}$ is another matrix with the same properties, then there exists a matrix $B\in \GL_n(\cH_\C)$ such that $Y'^{lr}=Y^{lr}B$ and $B(z+1)=B(z)$.
\end{theorem}

\begin{remark}
The statement in \cite{Birkhoff}, which is corrected in \cite{Borodin}, requires some hypotheses on $A$ in order to better understand the asymptotic growth of the solutions. We will not need this, so in return for proving a weaker property we can find solutions in full generality.
\end{remark}

\begin{definition}
Let $y(z+1)=A(z)y(z)$ be a difference equation as above. Let $Y^l(z)$ and $Y^r(z)$ be two solution matrices given by~\ref{thm:Borodin}. The \textbf{monodromy matrix} of the equation is the matrix $P(z) = (Y^r)^{-1}Y^l$. It is determined up to multiplication on the left and on the right by two holomorphic, invertible, periodic matrices, i.e. it is a well-defined element of $\GL_n(\cH_{\C^*})\backslash \GL_n(\M_{\C^*}) / \GL_n(\cH_{\C^*})$, where functions on $\C^*$ are pulled back to $\C$ via $z\mapsto e^{2\pi i z}$. Note that $P(z)$ is itself periodic.
\end{definition}

\begin{proposition}[{\cite[Theorem IV]{Birkhoff}}]\label{prop:monodromyIsDiagonal}
Let $u=e^{2\pi i z}$ be the coordinate in $\C^*$, and let $A(z)$ and $P(u)$ be as above. Then the class of $A$ in $\GL_n(\cH_{\C^*})\backslash \GL_n(\M_{\C^*}) / \GL_n(\cH_{\C^*})$ is represented by a diagonal matrix $\operatorname{diag}(d_i)\in \GL_n(\C(u))$ such that $d_i$ divides $d_{i+1}$, which is unique up to multiplication by diagonal matrix with entries in $\C[u,u^{-1}]^\times$.
\end{proposition}

\begin{proposition}\label{prop:relationToBorodin}
Consider $A(z),P(e^{2\pi i z})$ as above, and fix $p\in \C$. Let $M\subseteq \C(t)^n$ be the (holonomic) $\ddA$-module generated by any basis of $\C(t)^n$. Consider the composition $Q: \C((z-p))\otimes_{\C[[z-p]]} M|_\pz^l\to\C((z-p))\otimes M_p \to \C((z-p))\otimes M|_\pz^r$. There are $\C[[z-p]]$-bases of $M|_\pz^l$ and $M|_\pz^r$ such that the matrix of $Q$ equals $P$.
\end{proposition}

\begin{corollary}\label{cor:relationToBorodin}
Let $M$ be a holonomic $\ddA$-module. Then the collection of punctured local types $\{M|_\pzz\}_{p\in \C/\Z}$ and the monodromy matrix of $\C(z)\otimes M$ determine each other.
\end{corollary}

\begin{proof}[Proof of Theorem~\ref{thm:Borodin}]

Construct a difference module structure on $V=\C(t)^n$ using Construction~\ref{con:diffEq}. Fix $\mathbb H^l=\{z\in \A^1:\Re z<- N\}$ and $\mathbb H^r=\{z\in \A^1:\Re z> N\}$ two half-planes with the property that all the singularities of $A$ (i.e. the points where $A$ or $A^{-1}$ is not defined) have real parts contained in $(-N-1,N+1)$. On $\bH^l$, $\tau$ and $\tau^{-1}$ both preserve the analytic trivial bundle $\cH^n\subset V^{\an}$ (we will abuse notation and write $\tau$ in place of $\tau^{\an}$), so we have an $\cH$-linear isomorphism $t^*\circ \tau:\cH^n|_{\bH^l}\overset{\sim}\to (t^*\cH^n)|_{\bH^l}$, where $t$ is the translation. 
Therefore, $\cH^n$ has a $\Z$-equivariant structure, which can be extended past $\bH^l$ if we modify the vector bundle $\cH^n$. Consider the coherent sheaf $\mathcal Y^l\subset V^{\an}$ defined on any bounded open set $U$ by $\cY^l(U)=\{s\in V^{\an}:\tau^{-i} s\in \cH^n(U-i) \forall i\gg 0\}$. Since $\tau\cH^n|_{\bH^l}=\cH^n|_{\bH^l}$, this definition is the same as $\cY^l(U)=\{s\in V^{\an}:\tau^{-i_0} s\in \cH^n(U-i_0)\}$ where now $i_0$ is any integer such that $U\subseteq \bH^l + i_0$. From the definition it is clear that $\cY^l$ is locally free, since we have defined it to be locally isomorphic to $\cH^n$, and that it is $\tau$-invariant. Lastly, $\cY^l|_{\bH^l}=\cH^n|_{\bH^l}$ as subsheaves of $V^{\an}$. Similarly, we define $\cY^r$ to be $\tau$-invariant and coinciding with $\cH^n$ over $\bH^r$.

Now, the $\Z$-equivariant structure on $\cY^{lr}$ allows it to descend along the quotient $\pi:\C\to \C/\Z=\C^*$, which is given by $z\mapsto e^{2\pi i z}$. This is done by defining $\ov \cY^{lr}(U)=\{f\in \cY^{lr}(\pi^{-1}(U)):\tau f=f \}$. In particular, on a simply connected set $U$, $\pi^{-1}(U)=\bigsqcup_{i\in \Z} U_i$, and $\ov \cY^{lr}(U) = (\pi^{-1})^*\cY^{lr}(U_i)$, where $i$ can be chosen arbitrarily. Since locally $\ov \cY^{lr}$ is isomorphic to $\cY^{lr}$, it is also locally free, and therefore it is a trivial vector bundle, since $\C^*$ has no nontrivial analytic vector bundles \cite[Theorem 30.4]{forster2012lectures}.

Since $\ov \cY^{lr}$ is trivial, consider a basis and pull it back to $\cY^{lr}$: this yields a $\tau$-invariant basis of $\cY^{lr}$, i.e. a basis of meromorphic solutions to the equation $y(z)=\tau y(z)=A(z-1)y(z-1)$, or equivalently $y(z+1)=A(z)y(z)$. Further, since $\cY^l|_{\bH^l}=\cH^n|_{\bH^l}$, these solutions are holomorphic and they form a basis at every point of $\bH^l$, and similarly for $\cY^r$. Let the two matrices formed by these column vectors be $Y^l(z)$ and $Y^r(z)$.

Now, suppose we have another matrix $Y'^l$ with the same properties. The columns of $Y'^l$ are $\tau$-invariant, and when restricted to $\bH^l$, they lie in $\cY^l|_{\bH^l}=\cH_{\bH^l}^n$ (possibly by shrinking $\bH^l$ to a smaller half-plane). By the $\tau$-invariance of both $Y'^l$ and $\cY^l$, it follows that the columns of $Y'^l$ form a basis of $\cY^l$ at every point in $\C$, including outside of $\bH^l$. Also, since they are $\tau$-invariant, they descend to sections of $\ov \cY^l$, which form a basis at every point (because locally it is isomorphic to $\cY^l$), so they form a global basis for $\ov \cY^l$. Two bases for $\ov \cY^l\cong \cH_{\C^*}^n$ differ by a matrix $\ov B(u)\in \GL_n(\cH_{\C^*})$ (acting on the right, since we are working with column vectors). Therefore, $Y'^l(z)=Y^l(z) \pi^*\ov B(u)= Y^l(z)B(e^{2\pi i}z)$, as desired, and similarly for $Y'^r$.

\end{proof}



\begin{proof}[Proof of Proposition~\ref{prop:monodromyIsDiagonal}]
First of all, note that if $A(p)$ is well-defined and invertible at $p$, then on a neighborhood $B$ of $p$, $\tau$ gives an isomorphism $\cH^n|_B \cong \cH^n|_{B+1}$. Therefore, on the open set consisting of $\Z$-orbits that don't contain zeroes or poles of $A$, $\cY^l=\cY^r=\cH^n$. Therefore, on the image of this open set we have that $\ov\cY^l=\ov \cY^r$, so the matrix $P(u)$ mapping a basis of one to the other is holomorphic and invertible away from a finite set of points, the images of the zeroes and poles of $A$. 

The rest of the proof is very similar to the algorithm that computes the Smith normal form. The similarity is due to \cite[Theorem 15.15]{Rudin}, which ensures that every finitely generated ideal in the ring of holomorphic functions is principal. Note that this extends easily to every finitely generated fractional ideal (recall that $\M$ is the field of fractions of $\cH$, \cite[Theorem 15.12]{Rudin}): for such a fractional ideal $I=f_1/g_1\cH + \cdots + f_m/g_m \cH$ with $f_i,g_i\in \cH$, we have that $g_1\cdots g_mI$ is a principal ideal generated by some $f$, so $I$ is generated by an element of the form $f/(g_1\cdots g_m)$. Note that in particular any two meromorphic functions $f,g$ have a greatest common divisor $h$ such that $h\cH=f\cH+g\cH\subset \M$, and we have B\'ezout's identity: $af+bg=h$ for some $a,b\in \cH$.

Let $P=(p_{ij})$, and let $g$ be a generator of the fractional ideal $\sum_{i,j} p_{ij}\cH\subseteq \M$: since $P$ is defined away from a finite set, the entries $p_{ij}$ have a finite set of poles, so $g$ also has a finite set of poles. Note that multiplying by matrices in $\GL_n(\cH)$ doesn't change the ideal $(g)$. By permuting the rows, assume that $p_{11}\neq 0$. We will now ensure that $p_{i1}=0$ for every $i\neq 1$. Let $h$ be the greatest common divisor of $p_{11}$ and $p_{21}$, and take B\'ezout's identity $ap_{11}+bp_{21}=h$. We multiply on the left by the matrix which is the identity except for the top left $2\times 2$ block, which is given by:
\[
E = \matriz{a}{b}{-\frac{p_{21}}{h}}{\frac{p_{11}}{h}}.
\] 
After this multiplication, the $(2,1)$ term vanishes, while the $(1,1)$ term is replaced by a divisor. It is clear that we can carry out this procedure on all the remaining rows, so we may assume that $p_{i1}=0$ for $j\neq 0$. At the end of this process, $p_{11}$ has been replaced by a divisor $\wt p_{11}$, i.e. a function such that $\frac{p_{11}}{\wt p_{11}}\in \cH$.

At this point, the assumption that $A$ and $A^{-1}$ are well-defined away from a finite set imply that $p_{11}$ has a finite set of zeroes and poles. Now, continuing the Smith normal form algorithm, we carry out the steps in the previous paragraph on the transpose matrix, to ensure that $a_{1j}=0$ for $j\neq 1$. While doing this, the entries in the first column might become nonzero. However, $p_{11}$ is replaced by a divisor of $p_{11}$. Again, the new $p_{11}$ has a finite set of zeroes and poles.

We repeat this process with rows and columns, noting that each new $p_{11}$ must be a divisor of the previous one. However, it is always the case that $p_{11}\in g\cH $, so eventually it must stabilize: $g$ has a finite set of poles so $p_{11}/g$ is a holomorphic function with a finite set of zeroes, and this set gets smaller (with multiplicity) at very step. Once the process stabilizes, $p_{11}$ will divide every entry in its row and column, so we can perform row and column operations to ensure everything in the first row and column is a zero. We can continue inductively until we obtain a diagonal matrix.


Once we have a diagonal matrix, its entries have a finite set of zeroes and poles, so we can multiply by a diagonal matrix with holomorphic nonzero entries to ensure that all entries become rational functions. Now we have a matrix with rational entries, so the classical Smith normal form ensures that we can finish the algorithm so that each entry divides the next. The Smith normal form is unique up to multiplication by a diagonal matrix. It follows that this matrix is unique as well, up to multiplication by a diagonal matrix of rational functions. Such a diagonal matrix is invertible if it has entries in $\C[u,u^{-1}]^\times$.

\end{proof}

\begin{example}
Consider the following equation from \cite[Remark 1.5]{Borodin}:
\[
y(z+1)=\left(\begin{matrix}1 & 1/z\\0&1/e\end{matrix}\right)y(z).
\]
We have the following solutions:
\[
Y^l(z)= \matriz{1}{\frac{2\pi i}{e^{2\pi i z}-1}-\sum_{n=0}^\infty\frac{e^{-z-n}}{z+n}}{0}{e^{-z}};
Y^r(z)= \matriz{1}{-\sum_{n=0}^\infty\frac{e^{-z-n}}{z+n}}{0}{e^{-z}}.
\]
So the monodromy matrix is
\[
P(z)=\matriz{1}{\frac{2\pi i}{e^{2\pi i z}-1}}{0}{1} \sim       \matriz{\frac{1}{u-1}}{0}{0}{u-1}.                                              
\]
\end{example}

\begin{proof}[Proof of Proposition~\ref{prop:relationToBorodin}]


Let $L_M$ be the $\C[z]$-submodule generated by the chosen basis of $V$, which will be a free module of rank $n$. $L_M$ coincides with the trivial bundle away from a finite set. In particular, on some left half-plane it will coincide with $\cY^l$, which is also the trivial vector bundle. Since $\tau\cY^l=\cY^l$, we have that if $n\gg 0$ and $B$ is a ball around $p$,
\[
(\tau^n L_M^{\an})|_B=\tau^n(L_M^{\an}|_{B-n})\overset{n\gg 0}= \tau^n (\cY^l|_{B-n})=(\tau^n \cY^l)|_{B}=\cY^l|_{B}.
\]
And similarly, if $n\gg 0$, $(\tau^{-n} L_M^{\an})|_B=\cY^r |_{B}$. Also, note that the local type concerns formal fibers, which are the same for analytic and algebraic bundles. Therefore, we have the following commutative diagram for any $p\in \C$
\[
\begin{tikzcd}[column sep = 4 em,ampersand replacement=\&]
\ov \cY^l_{e^{2\pi i p}} \arrow[r]\arrow[d,"\pi^*","\sim"'] \& (\M_{\C^*}^n)_{e^{2\pi i p}}\arrow[d,"\pi^*","\sim"'] \& \ov \cY^r_{e^{2\pi i p}} \arrow[l]\arrow[d,"\pi^*","\sim"']\\
\cY^l_{p} \arrow[r,"\tau^n"]\arrow[d,equals] \& (V^{\an})_{p} \& \cY^r_{p} \arrow[l,"\tau^{-n}"']\arrow[d,equals]\\
(\tau^n L_M)_p=M|_\pz^l \arrow[r] \& M_p \arrow[u]\& M|_\pz^r=(\tau^{-n} L_M)_p. \arrow[l]\\
\end{tikzcd}
\]
We can think of every object in the above diagram as a module over $\C[[z-p]] = \C[[u-e^{2 \pi i p}]]$. After tensoring with $\C((z-p))$ every arrow becomes an isomorphism, so the arrows that go right to left can be inverted. The composition $\C((z-p))\otimes \ov \cY^l_{e^{2\pi i p}} \to \C((z-p))\otimes \ov \cY^r_{e^{2\pi i p}}$ is given by the matrix $P(u)$, so it agrees with the composition $Q$ as desired.
\end{proof}

\begin{proof}[Proof of Corollary~\ref{cor:relationToBorodin}]
Since $M$ is holonomic, $\C(z)\otimes M$ is finite dimensional over $\C(z)$, by Proposition~\ref{prop:finiteStalks}, so we may choose a basis of $\C(z)$ so that $\tau$ is given by a matrix $A(z)$. The punctured local types are nontrivial in a finite set of $\Z$-orbits (Lemma~\ref{lem:glattices}). Let $Q_1,\ldots Q_m$ be the Smith normal forms of the corresponding maps $ \C((z-p_i))\otimes M|_{U_{p_i}^*}^l\to \C((z-p_i))\otimes M|_{U_{p_i}^*}^r $, which by Proposition~\ref{prop:relationToBorodin} all coincide with $P(u)$ for the matrix $A$ (over the corresponding ring $\C[[z-p_i]]$). By the uniqueness of the Smith normal forms, it must follow that $P(u)=Q_1 \cdots Q_m$. Conversely, starting with $P(u)$, the matrix $Q_i$ at $p_i$ is given by clearing away all the factors in $P(u)$ that have no zeroes or poles at $e^{2\pi i p_i}$.
\end{proof}

\section{Local Mellin transform}

In this section we will show Theorem~\ref{thm:localMellin}. Throughout we will focus on the case of $\Phi^l_{p+\Z}$, since the corresponding case of $\Phi^r_{p+\Z}$ can be obtained by the symmetry as discussed in the introduction. Let $\D=\C[x^{\pm 1}]\langle \partial\rangle$ be the ring of differential operators on $\A^1\setminus\{0\}$, and let $\Dz=\C((x))\langle \partial\rangle$ be the ring of differential operators on the punctured disk at $0$. Recall that Theorem~\ref{thm:localMellin} amounts to defining $\M^{(0,p+\Z)}$ that fits into the following diagram:
\[
\begin{tikzcd}[column sep = 4 em,ampersand replacement=\&]
\Hol(\D)\arrow[r,"\M","\sim"'] \& \Hol(\ddA)\\
\Hol(\mathcal{D}_{K_0})^{\mathrm{reg},(p)}\arrow[from=u,"{}^{\reg,(p)}\circ \Psi_0"]\arrow[r,"\M^{(0{,}p+\Z)}","\sim"']\& \finmod.\arrow[from=u,"\Phi_{p+\Z}^l"]
\end{tikzcd}
\]
Here $\M$ is the Mellin transform, the equivalence induced by the ring isomorphism $\D\cong \ddA$ given by $x\mapsto \tau$ and $x\partial \mapsto z$. The functor of nearby cycles $\Psi_0$ in this case is the functor $\C((x))\otimes_{\C[x]}\bullet$, which we compose with picking the regular part with leading coefficient in $p+\Z$.

The classification of Levelt and Turrittin already ensures that $\Hol(\mathcal D_{K_0})\cong \finmod$, by an equivalence that maps $\C[\pi]/\pi\C[\pi]\in \finmod$ to the $D$-module $\Dz/\Dz(x\partial-p)$. However, we will not take this approach, and instead we will construct $\Mel$ by using properties of difference modules. We will use the functor $\iota_{p!}$ (denoted $\iota^\to_{p!}$ in the introduction), given by the following formula:
\[
\f{\iota_{p!}}{\finmod }{\Mod(\ddA)}{M}{\C((\tau))\otimes_{\C}M.}
\]
The $\C((\tau))$-module $\iota_{p!}M$ acquires the structure of a $\ddA$-module by letting $z$ act on a simple tensor $\left(\sum a_n\tau^n \right)\otimes m$ as follows:
\[
z\left(\sum_{n\ge -N} a_n\tau^n \right)\otimes m = ``\sum (a_n\tau^n \otimes (\pi+p+n)m)" = \left(\sum (p+n)a_n\tau^n\right) \otimes m + \left(\sum a_n\tau^n \right)\otimes \pi m .
\]
We will show that $\iota_{p!}$ is the right adjoint to $\Phi^l_{p+\Z}$. On the other side of the diagram, the right adjoint to ${}^{\reg,(p)}\circ \Psi_0$ is the forgetful functor to $\D$-modules, which we will denote $\jj_{0*}$. Consider then the diagram:
\[
\begin{tikzcd}[column sep = 4 em,ampersand replacement=\&]
\Mod(\D)\arrow[r,"\M","\sim"'] \& \Mod(\ddA)\\
\Hol(\mathcal{D}_{K_0})^{\mathrm{reg},(p)}\arrow[u,"\jj_{0*}"]\arrow[r,"\M^{(0{,}p+\Z)}","\sim"']\& \finmod.\arrow[u,"\iota_{p!}"]
\end{tikzcd}
\]
We will construct $\Mel$ by showing that the images of $\M\circ \jj_{0*}$ and $\iota_{p!}$ coincide. Both these functors are faithful, but not full, so we are referring to the their image as a subcategory of $\Mod(\ddA)$ which is not full. Once we how that the images coincide, the commutativity of the above diagram is automatic, and then Theorem~\ref{thm:formalCycles} follows because the $\iota_{p!}$ is the right adjoint to both $\M \circ \jj_{0*}\circ (\Mel)^{-1}$ and $\Phi^l_{p+\Z}$, so they must be canonically isomorphic because adjoints are unique.

\begin{remark}
It is possible to consider vanishing cycles on all $\Z$-orbits at once, by simply making $\Phi^l_{\fin}=\bigoplus_{p\in S} \Phi^l_{p+\Z}$, where $\Phi^l_{\fin}M$ becomes a $\C[z]$-module by identifying $z=\pi-p$ on each summand. The set $S$ can be chosen to be any class of representatives of $\A^1/\Z$, for example the complex numbers with real part in $[0,1)$. In this case the local Mellin transform will give an equivalence between $\Hol(\mathcal{D}_{K_0})^{\mathrm{reg}}$ and the category of finite length modules supported on $S$.
\end{remark}


\subsection{A different approach to vanishing cycles}

We will describe the image of $\iota_{p!}$, which we will denote $\ppholoOrbit$, and obtain an equivalence $\iota_{p!}:\finmod\longrightarrow \ppholoOrbit$. We use four categories of $\C[z]\langle \tau\rangle$-modules, starting with $\modOrbit$. $\pmodOrbit$ consists of objects which are limits of modules in $\modOrbit$, and then we describe corresponding categories of ``small'' modules.


\begin{definition}The category $\modOrbit$ is the full subcategory of left modules $V$ over $\ddA^l=\C[z]\langle \tau \rangle$ satisfying the following properties:
\begin{enumerate}
\item Any $m\in V$ is supported on $p+\Z$, i.e. there exists a $P(z)\in \C[z]$ such that $P(z)m=0$ and the roots of $P$ are contained in $p+\Z$.
\item $\tau:V\to V$ is a locally nilpotent map, i.e. for every $m$ there's a natural number $n$ such that $\tau^n m=0$.
\end{enumerate}
\end{definition}
\begin{definition}\label{def:finModule}A module in $\modOrbit$ is \textbf{holonomic} if $\tau^{-1}(0)$ is finite dimensional. We denote the category of holonomic modules by $\holoOrbit$.
\end{definition}

In Section \ref{sec:finmodules} we describe the relevant properties for these modules. We consider the collection of $\C[[\tau]]\langle z\rangle$-modules which are limits of these modules.

\begin{definition}\label{def:ProfinSupport}
The category $\pmodOrbit$ is defined to be the category of $\C[[\tau]]\langle z\rangle$-modules $M$ such that the following natural map is an isomorphism: 
$
M\to \lim_\gets M/L
$, where $L$ ranges across the set of $\C[[\tau]]\langle z\rangle$-submodules such that $M/L\in \modOrbit$.
\end{definition}
\begin{definition}
We say a module $M\in\pmodOrbit$ is \textbf{holonomic} if it is of Tate type, i.e. if it has an open finitely generated $\C[[\tau]]$-submodule $U$ such that $\tau^{-1}(U)/U$ is a finite dimensional vector space. The category of holonomic modules will be denoted $\pholoOrbit$. We let $\ppholoOrbit\subset \pholoOrbit$ be the full subcategory consisting of modules on which $\tau$ acts as a unit.
\end{definition}
Proposition \ref{prop:proHoloOverFin} gives several equivalent definitions to the definition above.
\begin{definition}
We define the functor \textbf{``sections with support at $p$"} $\iota^!_p:\pholoOrbit\rightarrow \Mod(\C[z])$ by taking
\[
\iota^!_p:M\longmapsto \iota^!_pM = \{m\in M:\exists n,(z-p)^nm=0 \}
.\]
Corollary~\ref{cor:iotaShriekWellDefd} shows that the image of $\iota_p^!$ is contained in $\finmod$.
\end{definition}

The remainder of this section is devoted to building the tools to prove the following proposition. Its proof can be found in Section~\ref{sec:proofofIotaEquiv}.

\begin{proposition}\label{prop:iotaEquiv}
The functors $\iota_p^!$ and $\iota_{p!}$ induce inverse equivalences $ \ppholoOrbit\longleftrightarrow\finmod$.
\end{proposition}

\begin{remark}
The above Proposition can be proven by taking the Mellin transform of modules in\linebreak$\ppholoOrbit$, which turns the difference modules into $D$-modules on the punctured formal disk. Then the classification of said differential operators can be used to obtain the result. However, we have chosen to take an alternative approach to the proof, which involves only difference modules.
\end{remark}

\subsubsection{Difference modules with support on an orbit}\label{sec:finmodules}

Let us show some useful properties about $\holoOrbit$.

\begin{remark}
All modules $V$ in $\modOrbit$ have a natural increasing filtration $V^i= \tau^{-i}(0) = \tau^{-1}(V^{i-1})$. Note that $\tau$ induces a map $V^i\to V^{i-1}$ with kernel $V^1$, which implies that $\dim V^i \le \dim V^{i-1} + \dim V^1$. We will use this notation in what follows.
\end{remark}

\begin{proposition}\label{prop:holoQuotients}
Both $\modOrbit$ and $\holoOrbit$ are closed under submodules, quotients and extensions. Further, if $W$ is a submodule or a quotient of $V\in \holoOrbit$, then $\dim W^1\le \dim V^1$.
\end{proposition}
\begin{proof}$\modOrbit$ is clearly closed under submodules, quotients and extensions.

Being holonomic is clearly preserved under submodules. For quotients, suppose $V\in \holoOrbit$ and $W$ is a submodule of $V$. Then the only nontrivial condition is that in $V/W$, $\tau^{-1}(0)$ is finite dimensional, i.e. that $\tau^{-1}(W)/W$ is finite dimensional. Note that $\tau^{-1}(W^i)\cap W = W^{i+1}$. Therefore, $\frac{\tau^{-1}(W)}{W} = \frac{\bigcup_i \tau^{-1}(W^i)}{W}=\bigcup_i \frac{\tau^{-1}(W^i)}{W^{i+1}}$. Since $\tau$ induces a map $\tau^{-1}(W^i)\to W^i$ with kernel contained in $V^1$,
\[
\dim \frac{\tau^{-1}(W^i)}{W^{i+1}} = \dim \tau^{-1}(W^i) -\dim W^{i+1}\le \dim \tau^{-1}(W^i) -\dim W^i \le \dim V^1
.\]
Therefore, $\tau^{-1}(W)/W$ is a union of subspaces of dimension at most $\dim V^1$, and therefore it has dimension at most $\dim V^1$, which implies that $V/W\in \holoOrbit$.

Finally, for extensions, observe that a short exact sequence $0\to U\to V\to W\to 0$ yields an exact sequence of vector spaces $0\to U^1 \to V^1 \to W^1$.
\end{proof}

We will use the following lemma in the sequel.

\begin{lemma}\label{lem:iotaFinHolo}Let $V\in \holoOrbit$. The following inequalities hold:
\begin{enumerate}
\item $\dim \iota^!_pV\le \dim V^1$.
\item $\dim V/\tau V\le \dim V^1$, with equality if and only if $V$ is finite dimensional.
\end{enumerate}
\end{lemma}
\begin{proof}\ 
\begin{enumerate}
\item Since $V^1$ is a torsion $\C[z]$-module, we may decompose it based on supports, as $V^1=\bigoplus_{i\in \Z}(V^1)_{p+i}$.
It is finite dimensional, so only a finite number of the components are nonzero. On $\iota^!_pV$, we may consider the filtration by subvector spaces $
(\iota^!_pV)^i=V^i\cap \iota^!_pV$. Then there are injections
$\tau^i:(\iota^!_pV)^{i+1}/(\iota^!_pV)^{i}\hookrightarrow (V^1)_{p+i}
$, which show that $
\dim \iota^!_pV = \sum_i \dim (\iota^!_pV)^{i+1}/(\iota^!_pV)^{i}\le \sum_i \dim (V^1)_{p+i} \le \dim V^1$, as desired.
\item Consider the exact sequence $0\to V^i \to V^{i+1}\overset{\tau^i}\longrightarrow V^1 \to V^1/(\tau^iV^{i+1})\to 0 $. The dimension of the last term is at most $\dim V_1$ and it is nondecreasing with $i$. This dimension equals $d_i = \dim V^1+\dim V^{i}-\dim V^{i+1}$. If it is ever the case that $d_i=\dim V^1$, this implies that $V^i=V^{i+1}=V$, so $V$ is finite dimensional. It remains to observe that the exact sequence $0\to V^1\to V^{i+1}\overset{\tau}\to V^i\to {V^i}/({\tau V^{i+1}}) \to 0$ yields the identity $\dim V^1+\dim V^{i}-\dim V^{i+1} = \dim V^{i}/(\tau V^{i+1})$.
In this identity, the limit of the right-hand side equals the dimension of $V/\tau V$, and it is at most $\dim V_1$, as desired.
\end{enumerate}
\end{proof}
\begin{corollary}
Every module $V\in \holoOrbit$ is Artinian.
\end{corollary}
\begin{proof}
Given a decreasing sequence $V_j\subset V$, we consider for very $j$ the sequence. $(a_j^i)_{i\in \N}=(\dim (V_j)^i)_{i\in \N}$. The proof of the previous lemma shows that $a_j^i$ is nondecreasing and concave when $i$ increases and $j$ is fixed: $a_j^{i+1}-a_j^i = a_j^1-d_i$ is nonincreasing. As $i$ is fixed and $j$ varies $a_j^i$ it nonincreasing.

Let us show that for such a sequence of nonnegative integers if $a_j^{i+1}\ge a_j^{i}$, $a_{j+1}^i\le a_j^i$ and $a_j^{i+1}-a_j^{i}\le a_j^i-a_j^{i-1}$, then there is some big enough $N$ such that if $j\ge N$, $a_j^i=a_{N}^i$ for all $i$. Consider the quantity $k_j=\min_i a_j^{i+1}-a_j^{i}$. Both $k_j$ and $a_j^0$ are nonincreasing, so there is some $N$ for which $k_j$ and $a_j^0$ are constant if $j\ge 0$. Let us forget about $a_j^i$ for $j<N$, since we only care about the eventual stabilization. Consider now $b_j^i=a_j^i-a_j^0-ik_j$. We have that
\[
\begin{array}{cc}
b_j^i \ge 0 ;& b_j^{i+1}-b_j^i = a_j^{i+1}-a_j^i -k_j\ge 0;\\
b_{j+1}^i\le b_j^i; & b_j^{i+1}-2b_j^i+b_j^{i-1} = a_j^{i+1}-2a_j^i+a_j^{i-1} \le 0.
\end{array}
\]
So the new sequence $(b_j^i)$ keeps the same properties, $b_j^0=0$ and when $j$ is fixed, $b_j^i$ is eventually constant. Now, $m_j=\max_i b_j^i$ is also nonincreasing, so it is eventually constant as well. As before, let us ignore the small enough $j$'s so that $m_j$ is not the minimum value, so we have that $b_j^i\le m$.

Consider the sequence $(c^i)=(\min_j b_j^i)$, and notice that $(c^i)$ is also positive, nondecreasing, concave and bounded above by $m$. Let $i_0$ be such that $c^{i_0}=c^{i_0+1}$. There is some $j$ for which $c^{i_0}=b_j^{i_0}=b_j^{i_0+1}$, so $c^{i_0}=m$. Therefore, we have that $b_j^i=m$ for every $i\ge i_0$. There are only finitely many functions $[0,i_0]\to [0,m]$, so the sequence eventually stabilizes as desired.

\end{proof}

\subsubsection{Limits of difference modules with support on an orbit}

We are particularly interested in the category $\pholoOrbit$. A way to state some of the properties of these modules is by using the natural topology on them. Recall the definition of the $\tau$-adic topology.

\begin{definition}
Let $M$ be a $\C[\tau]$-module. The \textbf{$\tau$-adic topology} on $M$ is defined as follows: a subspace $U\subset M$ is open if for every finitely generated submodule $N\subseteq M$, there is a $k$ such that $\tau^kN\subseteq U$. Open subspaces form a basis of neighborhoods of $0$.
\end{definition}
\begin{remark}
A $\C[[\tau]]\langle z\rangle$-module $M$ is in $\pmodOrbit$ if and only if there is a basis $\{U\}$ of open $\C[[\tau]]\langle z\rangle$-modules such that for every $U$ and every $s\in M/U$, there is a polynomial $P(z)$ such that $P(z)s=0$ and the roots of $P$ are contained in $p+\Z$. This is due to the fact that on a $\C[[\tau]]$-module the $\tau$-adic topology is always Hausdorff and complete, so $M\to \displaystyle\lim_{L\text{ open}} M/L$ is always an isomorphism.
\end{remark}


\begin{proposition}\label{prop:proHoloOverFin}
For a module $M\in \C[[\tau]]\langle z\rangle$, the following are equivalent:
\begin{enumerate}
\item $M\in \pholoOrbit$.
\item $M$ contains an open $\C[[\tau]]\langle z\rangle$-submodule $N$ that is finitely generated over $\C[[\tau]]$ and such that $M/N\in \holoOrbit$. Moreover, $z-c$ acts as a unit on $M$ for any $c\notin p+\Z$.
\item There is a basis $\{N_i\}$ of open neighborhoods of $0$ which are $\C[[\tau]]\langle z\rangle$-submodules such that $M/N_i\in \holoOrbit$ and $\dim \tau^{-1}N_i/N_i$ is bounded.
\item There is a basis $\{N_i\}$ of open neighborhoods of $0$ which are $\C[[\tau]]\langle z\rangle$-submodules that are finitely generated over $\C[[\tau]]$ and such that $M/N_i\in \holoOrbit$.

\end{enumerate}
\end{proposition}


\begin{proof}\ 
\begin{enumerate}
\item[$1\Rightarrow 2$] Let $U\subseteq M$ be a witness to $M$ being a space of Tate type, i.e. $U$ is a finitely generated $\C[[\tau]]$-module and $\tau^{1}(U)/U$ is finite dimensional. By the assumption that $M\in \pmodOrbit$, $M$ has a particular basis of open submodules. We may choose an $N\subseteq U$ that is an open $\C[[\tau]]\langle z\rangle$-submodule, and $M/N\in \modOrbit$. In fact, $M/N\in \holoOrbit$: as a $\C[[\tau]]$-module, $M/N$ is an extension of $M/U$ by $U/N$, and there is an exact sequence
\[
0\longrightarrow \frac{\tau^{-1}(N)\cap U}{N}\longrightarrow \frac{\tau^{-1}(N)}{N}\longrightarrow \frac{\tau^{-1}(U)}{U}
.\]
The last term in the sequence is finite dimensional by assumption. The first one is contained in the torsion finitely generated $\C[[\tau]]$-module $U/N$, and therefore it is also finite dimensional. This shows that $M/N\in \holoOrbit$. Since $N\subseteq U$, $N$ is finitely generated over $\C[[\tau]]$.

Lastly, if $c\notin p+\Z$, $z-c$ acts as a unit because $M$ is an inverse limit of modules in $\modOrbit$, on each of which $z-c$ acts as a unit.

\item[$2\Rightarrow 3$] Let $N\subseteq M$ be the submodule in the assumption. Then we claim that $\{\tau^iN\}_{i\ge 0}$ is the required basis. It is easy to check that they are $\C[[\tau]]\langle z\rangle$-submodules, given that $N$ is one. Let us see that $M/\tau^iN\in \holoOrbit$. For each $i$, consider the exact sequence
\[
0\longrightarrow 
\frac{N\cap \tau^{-1}\tau^iN}{\tau^iN}\longrightarrow \frac{\tau^{-1}\tau^iN}{\tau^i N}\longrightarrow \frac{\tau^{-1}N}{N}
.\]
The dimension of $\frac{N\cap \tau^{-1}\tau^iN}{\tau^iN}$ is bounded above. This is true because $N$ is a finitely generated $\C[[\tau]]$-module and this claim can be checked by writing $N$ as a direct sum of a finite module and a free module. Therefore, %
$\dim \tau^{-1}\tau^iN/\tau^i N$ is a bounded number. Let us show now that the $\tau^iN$ form a basis of open sets. They are indeed open since $N$ is. If $U$ is any other open subspace, then the fact that $N$ is finitely generated combined with the definition of an open set shows that $\tau^iN\subseteq U$ for a big enough $i$.

Lastly, let us show that $M/\tau^iN\in \holoOrbit$ for every $i$. It only remains to show that all its elements are supported on $p+\Z$. We consider the short exact sequence
\[
0\longrightarrow\frac{N}{\tau^iN}\longrightarrow \frac{M}{\tau^iN}\longrightarrow \frac{M}{N}\longrightarrow 0
.\]
The first term in the sequence is finite dimensional, so the support of its elements is in $p+\Z$, since $z-c$ acts as a unit on it for $c\notin p+\Z$. Therefore, $\supp M/\tau^iN\subseteq \supp N/\tau^iN\cup \supp M/N\subseteq p+\Z$.

\item[$3\Rightarrow 4$] Let $B=\{N_i\}$ be the basis in the statement. We must find a basis of finitely generated $\C[[\tau]]$-modules with the required properties. It will be a subset of $B$, namely we will choose a fixed $N\in B$, and the new basis will consist of the elements of $B$ contained in $N$. By Proposition \ref{prop:holoQuotients}, the quantity $\dim \tau^{-1}(N_i)/N_i$ is nondecreasing as $N_i\in B$ gets smaller. Let $N\in B$ be such that $\dim \tau^{-1}(N)/N =d$ is the maximum among elements $N_i\in B$, so in particular if $N_i\subseteq N$, $\dim \tau^{-1}N_i/N_i = d$. We will show that $N$ is finitely generated over $\C[[\tau]]$. From here, it follows that $B_N=\{N_i\in B:N_i\subseteq N\}$ is the required basis.

Let $N_i\in B_N$ and let $\pi_i:M\to M/N_i$ be the projection. Consider the short exact sequence: $
0\to \pi_iN \to \pi_iM \to \frac{M}{N} \to 0
$. Now, $V^1=\operatorname{Tor}_1^{\C[[\tau]]}(\C[[\tau]]/\tau\C[[\tau]],V)$, which can be seen using the free resolution $\C[[\tau]]\overset{\tau}\to \C[[\tau]]$. The sequence above induces the following Tor exact sequence of finite dimensional vector spaces
\[
0\longrightarrow \left(\pi_iN\right)^1 \longrightarrow \left(\pi_iM\right)^1 \longrightarrow \left(\frac{M}{N}\right)^1 \longrightarrow \frac{\pi_iN}{\tau\pi_iN}
.\]
The assumption that $\dim \tau^{-1}N/N$ is maximal implies that the dimension of the two spaces in the middle is equal, which in turn implies that $\dim \frac{\pi_iN}{\tau\pi_iN}\ge \dim (\pi_iN)^1$. By Lemma \ref{lem:iotaFinHolo} this implies that $\pi_iN$ is finite dimensional, and by Nakayama's lemma it is generated by any system of generators for $\pi_iN/\tau\pi_iN$. Further, $\dim \pi_iN/\tau\pi_iN\le \dim (\pi_iN)^1\le \dim (\pi_i M)^1 = d$.

Consider now the short exact sequences
\[
0\longrightarrow \frac{\tau^{-1}(N_i)\cap N}{N_i}\longrightarrow \pi_iN\overset{\tau}\longrightarrow \pi_iN\longrightarrow \frac{\pi_iN}{\tau\pi_iN}\longrightarrow 0
.\]
We claim that the inverse limit of these sequences as $i\to \infty$ is also exact. We can check the Mittag-Leffler conditions and then apply the results on exactness of inverse limits (\cite[\href{http://stacks.math.columbia.edu/tag/0598}{Tag 0598}]{stacks-project}). Splitting the exact sequence into two short exact sequences, we have that the Mittag-Leffler conditions hold because the spaces $\tau^{-1}(N_i)/N_i$ are finite dimensional, and because the maps $\tau \pi_iN\to \tau \pi_{i'}N$ are surjections, respectively. Therefore, the limit of the sequences is the exact sequence
\[
0\longrightarrow \tau^{-1}(0)\cap N\longrightarrow N\overset{\tau}\longrightarrow N\longrightarrow \frac{N}{\tau N}\longrightarrow 0
.\]
In particular, $N/\tau N = \lim_{\gets}  \pi_iN/\tau\pi_iN$. On the right-hand side we have an inverse limit of surjections of finite dimensional vector spaces of dimension at most $d$, so all the maps are eventually isomorphisms. Lifting any given basis for $N/\tau N$ will generate all the modules $\pi_iN/\tau\pi_iN$, so by Nakayama's lemma it will generate all the $\pi_iN$'s (we can apply the Lemma since these are finitely dimensional), and therefore it will generate $N$.

This shows that there is an element $N$ in the basis $B$ which is finitely generated over $\C[[\tau]]$. Therefore, the basis $B_N=\{N_i\in B:N_i\subseteq N\}$ satisfies the required properties.

\item[$4\Rightarrow 1$] This is clear.

\end{enumerate}

\end{proof}

\begin{corollary}\label{cor:proHoloOverFin}
Let $M\in \pholoOrbit$, and let $U\subseteq M$ be a open finitely generated $\C[[\tau]]$-module (which is guaranteed to exist by the definition). Then $U$ contains an open $\C[[\tau]]\langle z\rangle$-module $L$, such that
\begin{enumerate}[a)]
\item $L$ is finitely generated over $\C[[\tau]]$.
\item $\{\tau^i L\}_{i\ge 0}$ is a basis of neighborhoods of $0$.
\item $M/\tau^i L\in \holoOrbit$.
\item $\dim \frac{\tau^{-1}\tau^iL}{\tau^i L}=d$ is constant.
\item For any open $\C[[\tau]]\langle z\rangle$-submodule $N\subseteq M$, $M/N\in \holoOrbit$ and $\dim \tau^{-1}N/N\le d$.
\end{enumerate}
\end{corollary}
\begin{proof}
Let $B$ be a basis as in Proposition~\ref{prop:proHoloOverFin}, part 3), and choose $L\in B$ contained in $U$, such that $d=\dim \tau^{-1}L/L$ is the maximum in $B$. Let $L$ be chosen so that $\tau^{-1}L/L$ is the maximum among the basis provided by the Proposition. Since $L$ is contained in $U$, it satisfies a). Part b) is true for any open finitely generated $\C[[\tau]]$-module. For every $i$, let $N_i\in B$ be contained in $\tau^iL$. Then $M/\tau^iL$ is a quotient of $M/N_i\in \holoOrbit$, so we have part c). Finally, Proposition~\ref{prop:holoQuotients} ensures that $\dim \tau^{-1}\tau^iL/\tau^iL\ge \dim \tau^{-1}L/L$, and $\dim \tau^{-1}\tau^iL/\tau^iL\le \dim \tau^{-1}N_i/N_i\le d$, so we have part d). Finally, if $N$ is such a submodule, by the basis property there is some $\tau^iL$ contained in $N$. Again Proposition~\ref{prop:holoQuotients} ensures that $\dim \tau^{-1}N/N\le \dim  \tau^{-1}\tau^iL/\tau^iL =d$, so e) is satisfied.
\end{proof}

\begin{corollary}Both $\pmodOrbit$ and $\pholoOrbit$ are abelian categories.
\end{corollary}
\begin{proof}
They are both full subcategories of the abelian category $\Mod(\C[[\tau]]\langle z\rangle)$, so it is enough to show that they are closed under quotients and submodules. Let us start with $\pmodOrbit$: consider a short exact sequence of $\C[[\tau]]\langle z\rangle$-modules $0\to N\to M\to M/N\to 0$, and suppose $M\in \pmodOrbit$. For any $L\subseteq M$ such that $M/L\in \modOrbit$, we have a short exact sequence $0\to N/(N\cap L)\to M/L \to M/(N+L)\to 0$, which shows that both the submodule and the quotient are in $\modOrbit$. Taking the limits of these short exact sequences shows that $N,P\in \pmodOrbit$ (note that in this case the limit is exact because it is a limit of surjections).

Now, suppose we have a short exact sequence $0\to L\to M\to M/L\to 0$ in $\pmodOrbit$ and suppose $M\in \pholoOrbit$. Choose a basis $\{N_i\}$ of neighborhoods of $0\in M$ as in Corollary~\ref{cor:proHoloOverFin}, part 4). Then $\{N_i\cap L\}$ and $\{N_i+L/L\}$ are bases for $L$ and $P$ respectively, and it is straightforward to check that they have property 4) in Proposition~\ref{prop:proHoloOverFin} as well.

\end{proof}

\begin{corollary}\label{cor:iotaShriekWellDefd}
If $M\in \pholoOrbit$, then $\iota^!_pM\in \finmod$.
\end{corollary}
\begin{proof}
We must show that $\iota_p^!M$ is finite dimensional. Consider a basis of submodules $N_i\subseteq M$ as in Corollary~\ref{cor:proHoloOverFin}. For every $i$, $M/N_i\in \holoOrbit$, and Lemma~\ref{lem:iotaFinHolo} implies that $\dim\iota_{p}^!(M/N_i)\le \dim \tau^{-1}N_i/N_i$ is finite dimensional and bounded. Since $M\cong \lim_{\gets} M/N_i$, we can see that $\iota^!_p M\subseteq \lim \iota^!_p M/N_i$, so it is finite dimensional as desired.

\end{proof}
\subsubsection{Proof of Proposition~\ref{prop:iotaEquiv}}\label{sec:proofofIotaEquiv}

Finally, we have all the tools to prove Proposition~\ref{prop:iotaEquiv}. Let us abbreviate $\iota_!=\iota_{p!}$ and $\iota^!=\iota_p^!$.

\begin{lemma}\label{lem:ExceptionalInvImgExact}
The functor $\iota^!:\pholoOrbit\to \finmod$ is exact.
\end{lemma}
Using this lemma, we will now prove the proposition. First of all, it is straightforward to check that $\iota^!\iota_!\cong \Id$, so it remains to show that $\iota_!\iota^! \cong \Id$. There is a natural map $\eta:\iota_!\iota^!M\longrightarrow M$. Let us show that $\eta$ is an isomorphism. By the exactness of $\iota^!$, $\iota^!(\operatorname{co})\ker \eta \cong (\operatorname{co})\ker \iota^! \eta$. However, the isomorphism $\iota^!\iota_!\cong \Id$ implies that $\iota^!\eta$ is an isomorphism, so $\iota^!\ker \eta=\iota^!\coker \eta=0$. It is therefore enough to show that $\iota^!P=0$ implies that $P=0$.

Let us show that $\iota^!P=0$ for $P\in \ppholoOrbit$ implies that $P=0$. Let $L\subseteq P$ be as in Corollary~\ref{cor:proHoloOverFin}. Since $\iota^!$ is exact, $\iota^!P$ surjects onto $\iota^!(P/L)$. If $\iota^!P=0$, then $\iota^!(P/L)=0$. If $P/L\neq 0$, then it has some nonzero element $m$ supported at some point $p-j$ for some $j$. Then $\tau^jm$ is a nonzero element of $\iota^!(P/\tau^jL)$, contradicting the assumption that $\iota^!P=0$ and therefore $\iota^!(P/\tau^jL)=0$. Therefore, $P/L=0$. This implies that $P=L$ is a $\C((\tau))$-vector space which is finitely generated over $\C[[\tau]]$, so indeed $P=0$.

\begin{proof}[Proof of Lemma \ref{lem:ExceptionalInvImgExact}]
In general, $\iota^!$ is left exact. In order to show that it is right exact as well, it will be enough to show that it maps surjections to surjections. Let us start by considering a surjection $f:M\to N$ in $\holoOrbit$. Since $M,N\in \holoOrbit$, we may write $M=\bigoplus_{i\in \Z} M_{p+i}$ and similarly for $N$. The morphism $f$ sends each component $M_{p+i}$ to $N_{p+i}$, so $\iota^!f$ becomes the map $\iota^!f:M_p\longrightarrow N_p$, which must necessarily be surjective.

Now let us consider the case where $M,N\in \pholoOrbit$. Suppose we have a basis $\{L\}$ of open $\C[[\tau]]\langle z\rangle$-submodules in $M$. First we need to prove that the following map is an isomorphism:
\[\iota^!M=
\iota^!\lim_\gets \frac{M}{L} \longrightarrow 
\lim_\gets\iota^! \frac{M}{L} 
.\]
It can be checked that it is injective. To see that it is surjective, a system of compatible elements $\{s_L\}$ on the right-hand side corresponds to an element $s$ of $M$, and we must show that this element is torsion. Since the $L$'s are open, Corollary~\ref{cor:proHoloOverFin} ensures $M/L\in \holoOrbit$ and that $\dim \tau^{-1}L/L\le d$ for some fixed $d$. Further, by Lemma \ref{lem:iotaFinHolo}, $\dim \iota^!M/L\le \dim \tau^{-1}L/L\le d$. In particular, $(z-p)^d\iota^!M/L=0$, so $s_L$ is annihilated by $(z-p)^d$ for every $L$. This implies that $(z-p)^ds=0$, so $s\in \iota^!M$ as we wished.

Suppose now that we have a surjection $f:M\longrightarrow N$ in $\pholoOrbit$. We must show that the corresponding map $\iota^!f:\iota^!M\to \iota^!N$ is surjective. Let $\{L\}$ be a basis of neighborhoods of $0$ for $M$ as above. Further, since $f$ is a surjection, it can be checked that $\{f(L)\}$ is a basis of neighborhoods of $0$ in $N$ with the same properties. Thus $\iota^!f$ can be seen as a map
\[
\iota^!f:\iota^!\lim_{\gets} \frac{M}{L}\longrightarrow \iota^! \lim_{\gets} \frac{N}{f(L)} 
.\]
By the discussion above, the map is isomorphic to
\[
\lim_\gets\iota^!f_L:\lim_{\gets} \iota^!\frac{M}{L}\longrightarrow  \lim_{\gets} \iota^!\frac{N}{f(L)}
.\]
Each of the maps in the limit is surjective. A sufficient condition for an inverse limit of surjective maps to be surjective is the arrows forming the limit being surjections themselves \cite[\href{http://stacks.math.columbia.edu/tag/0598}{Tag 0598}]{stacks-project}. This is the case, because we have already shown that $\iota^!$ is right exact when restricted to $\holoOrbit$. This shows that $\iota^!$ is exact.

\end{proof}

\subsubsection{The right adjoint to vanishing cycles}

\begin{proposition}\label{prop:adjoints}
The functors $
\Phi_{p+\Z}^l$ and $\iota_{p!}
$ are adjoints in the following sense: if $M\in \Hol(\ddA)$ and $N\in \finmod$, there is a natural isomorphism
\[
\Hom_{\C[\pi]}(\Phi_{p+\Z}^l M,N)\cong\Hom_{\ddA}(M,\iota_{p!}N).
\]
\end{proposition}
\begin{remark}
Technically, it is not true that $\Phi^l_{p+\Z}\vdash \iota_{p!}$ because the image of $\iota_{p!}$ is not made of holonomic modules. However, the statement above is enough for our purposes. Notice that it implies that $\Phi^l_{p+\Z}$ is determined by this adjunction.
\end{remark}

\begin{proof}[Proof of Proposition \ref{prop:adjoints}]
Let $M\in \Hol(\ddA)$, and let $N\in \finmod$, and let us abbreviate $\Phi=\Phi^l_{p+\Z}$ and $\iota_!=\iota_{p!}$. We must find a natural isomorphism
\[
 \Hom_{\C[\pi]}(\Phi M,N)\cong\Hom_{\C[z]\langle\tau,\tau^{-1}\rangle}(M,\iota_!N)
.\]
First of all, note that since $\tau$ acts as a unit on both $M$ and $\iota_!N$, the forgetful functor gives an isomorphism $
\Hom_{\C[z]\langle\tau,\tau^{-1}\rangle}(M,\iota_!N)\cong\Hom_{\C[z]\langle\tau\rangle}(M,\iota_!N)$.

Throughout this proof, we will denote $\frac{\C[\tau,\tau^{-1}]}{\tau^{n+1}\C[\tau]}=\tau^{n}\C[\tau^{-1}]$ for short. In other words, we have that
\[
\C((\tau)) = \lim_{\sstack{\gets}{n\to \infty}}\frac{\C[\tau,\tau^{-1}]}{\tau^{n+1}\C[\tau]}=\lim_{\gets}\tau^{n}\C[\tau^{-1}]
.\]
Where the projection maps $\tau^n\C[\tau^{-1}] \to \tau^{n-1}\C[\tau^{-1}]$ are implied. Using this notation, we will also abbreviate $\tau^{n}\C[\tau^{-1}]\otimes_{\C} N$ to $\tau^{n}\C[\tau^{-1}]N$. By the definition of limit, we have that
\[\Hom_{\C[z]\langle\tau\rangle}(M,\iota_!N)\cong \Hom_{\C[z]\langle\tau\rangle}\left(M,\lim_{\gets}\tau^{n}\C[\tau^{-1}]N\right)\cong 
\lim_{\gets} \Hom_{\C[z]\langle\tau\rangle}\left(M,\tau^{n}\C[\tau^{-1}]N\right)
.\]
Consider now one of the arrows in the right-hand side limit:
\[
\f{\pi}{\Hom\left(M,\tau^{n+1}\C[\tau^{-1}]N\right)}{\Hom\left(M,\tau^{n}\C[\tau^{-1}]N\right)}{f}{\pi(f)=f\mod \tau^n N.}
\]
The homomorphism $\pi$ has an inverse: $\pi^{-1}(f)=\tau \circ f\circ \tau^{-1}$, where $\tau$ is seen as the $\C$-linear isomorphism $\tau^{n}\C[\tau^{-1}]N\to \tau^{n+1}\C[\tau^{-1}]N $. One verifies that $\pi^{-1}(f)$ is indeed $\C[z]\langle \tau\rangle$-linear, and that $\pi$ and $\pi^{-1}$ are inverses. Therefore, $\lim_{\gets} \Hom\left(M,\tau^{n}\C[\tau^{-1}]N\right)$ is a limit of a system of isomorphisms, so it is isomorphic to any one of its terms:
\[
\Hom_{\C[z]\langle\tau,\tau^{-1}\rangle}(M,\iota_!N)\cong\lim_{\gets} \Hom_{\C[z]\langle\tau\rangle}\left(M,\tau^{n}\C[\tau^{-1}]N\right) \cong \Hom_{\C[z]\langle\tau\rangle}\left(M,\C[\tau^{-1}]N\right)
.\]
We can write a map $f:M\to \C[\tau^{-1}]N$ as $f(s)  = \sum_{i\ge 0} \tau^{-i}\phi_i(s)$, where $\{\phi_i\}$ is a collection of maps $M\to N$. The conditions of $f$ being $\C[z]\langle\tau\rangle$-linear and the image of $f$ landing in $\C[\tau^{-1}]N$ rather than in $\C[[\tau^{-1}]]N$ boil down to the following three conditions:
\[
\begin{array}{crcl}
\forall m\in M & \phi_i(m) &=& \phi_0(\tau^i m)\\
\forall m\in M &\phi_0(zm)&=&z\phi_0(m)\\
\forall m\in M\exists n\in \Z_{\ge 0}& 0&=&\phi_n(m)=\phi_0(\tau^n m ).
\end{array}
\]
The first two conditions imply that $f$ is determined by a $\C[z]$-linear map $\phi_0:M\to N$, i.e.
\[
\Hom_{\C[z]\langle\tau\rangle}\left(M,\C[\tau^{-1}]N\right) \cong \left\{
\phi\in \Hom_{\C[z]} \left(M,N\right):
\forall m\exists n,\phi(\tau^nm)=0
\right\}=:\Hom_{\C[z]}^{(\tau)} \left(M,N\right).
\]
To finish the proof, we will show that the maps in $\Hom_{\C[z]}^{(\tau)} \left(M,N\right)$ are precisely the maps that factor through the map $M\to \Phi M$. Let $\phi\in \Hom_{\C[z]}^{(\tau)} \left(M,N\right)$, and let $L\in \Gl(M)$. Since $L$ is finitely generated, it follows that for some big enough $n$, $\phi(\tau^nL)=0$. Therefore, the map $\phi_p:M_p\to (\C[\tau^{-1}]N)_p\cong N$ sends $(\tau^nL)_p$ to $0$, and therefore it factors (uniquely) through a map $
\wt\phi:\Phi M = M_p/(\tau^nL)_p \longrightarrow N$.

To go in the opposite direction, let $g:\Phi M\to N$ and consider the composition $\wt g=g\circ \pi$, where $\pi$ is the projection $M\to \Phi M$. Taking $L\in \Gl(M)$ containing a given $m$, $(\tau^n m)_p\in (\tau^nL)_p$, so if $n\gg 0$, $\pi m=0$. Therefore, $\wt g\in \Hom_{\C[z]}^{(\tau)} \left(M,N\right)$. Putting all the steps together, we have concluded the proof.
\end{proof}

\subsection{Local Mellin transform}\label{sec:MellinTransform}

%
%
%
%

\begin{definition}
Let $F\in \Hol(K_0)^{\reg,(p)}$. The \textbf{local Mellin transform} of $F$ is defined as
\[
\Mel F = \iota_p^! (\M (\jj_{0*} F)).
\]
Where $\jj_{0*}$ is the forgetful functor $\Hol(\Dz)\to \Hol(\D)$ (we are using \cite{A}'s notation here). The vector space $\M \jj_{0*}F$ equals $F$, together with an action of $\C((\tau))\langle z\rangle$ given by $\tau^{\pm 1}\mapsto x^{\pm 1}$ and $z \mapsto x\partial$ and in Proposition~\ref{MellinIsEquivalence} we show that $\M \jj_{0*}F\in \pholoOrbit$, so it makes sense to apply $\iota_p^!$ to it.
\end{definition}

\begin{remark}
By definition, the $x$-adic topology on $\jj_{0*}F$ coincides with the $\tau$-adic topology on\linebreak$\iota_{p!}\Mel (F)$, and this together with the condition
\[
\M(\jj_{0*} (F))\overset{\sim}\longrightarrow \iota_{p!} (\Mel (F)).
\]
determines $\Mel $. This follows from the fact that $\iota_p^!$ and $\iota_{p!}$ are mutual inverses, Proposition \ref{prop:iotaEquiv}.

\end{remark}

\begin{proposition}\label{MellinIsEquivalence}
The functor $\Mel$ induces an equivalence
\[
\Mel :\Hol(\Dz)^{\mathrm{reg},(p)}\overset{\sim}\longrightarrow \finmod.
\]
\end{proposition}
\begin{proof}
Using Proposition \ref{prop:iotaEquiv}, this will follow from showing that the following functor is an equivalence, since it remains to compose with $\iota^!_p$:
\[
\M\circ \jj_{0*}: \Hol(\Dz)^{\mathrm{reg},(p)}\overset{\sim}\longrightarrow \ppholoOrbit.
\]
First of all, we must check that the image of $\Hol(\Dz)^{\mathrm{reg},(p)}$ is indeed contained in $\ppholoOrbit$. Let $V\in \Hol(\Dz)^{\mathrm{reg},(p)}$. By definition of the leading coefficient and having regular singularities, we can find a lattice $L\subset V$ such that $(x\partial-p)^n L\subseteq xL$ for some big enough $n$. This implies that $\M \jj_{0*} V/L\in \modOrbit$: if we let $s\in V$, then for some $m$, $x^ms\in L$. Denoting $\M\jj_{0*}s = \wh s$ and $\M\jj_{0*}L = \wh L$, we have that
\[
\tau^m(z-p+m)^n\wh s=(z-p)^n\tau^m\wh s\in \tau \wh L\Rightarrow \tau^{m-1}(z-p+m)^n\wh s\in \wh L \Rightarrow \]
\[
\Rightarrow \tau^{m-1}(z-p-m+1)^n(z-p+m)^n\wh s=(z-p)^n\tau^{m-1}(z-p+m)^n\wh s\in (z-p)^n\wh L\subseteq \tau \wh L \Rightarrow\]\[
\Rightarrow 
\tau^{m-2}(z-p-m+1)^n(z-p+m)^n\wh s \in \wh L \Rightarrow
(\cdots)\Rightarrow (z-p+m)^n\cdots (z-p+1)^n\wh s\in \wh L.
\]
A similar computation for $x^iL$ shows condition (4) in Proposition \ref{prop:proHoloOverFin}, so it follows that $\M \jj_{0*} V\in \pholoOrbit$.

It is clear that $\M\circ \jj_{0*}$ is fully faithful, since morphisms on both sides are $\C((x))\langle\partial\rangle \cong \C((\tau))\langle z\rangle$-linear maps. To show that it is essentially surjective, we just need to produce a preimage for every isomorphism class in $\ppholoOrbit$, or equivalently, for every module of the form $\iota_{p!} M = \C((\tau))\otimes_{\C} M$, where $M\in \finmod$. We may view $\iota_{p!} M$ as a module over $\C((x))\langle \partial \rangle$ via the Mellin transform. It is finite dimensional, since its dimension over $\C((x))$ equals $\dim_\C M$, and further it is regular and its leading coefficient is $p$, since it contains the lattice $L= \C[[x]]\otimes_{\C} M$ which is a witness to both these facts.

\end{proof}
The corollary below follows directly from the adjunctions $\Psi_0\vdash \jj_{0*}$ and $\Phi_{p+\Z}^l\vdash \iota_{p!}$.
\begin{corollary}\label{cor:localMellin}
The following square is commutative up to a natural isomorphism:
\[
\begin{tikzcd}[column sep = 4 em,ampersand replacement=\&]
\Hol(\D)\arrow[r,"\M","\sim"'] \& \Hol(\ddA)\\
\Hol(\mathcal{D}_{K_0})^{\mathrm{reg},(p)}\arrow[from=u,"{}^{\reg,(p)}\circ \Psi_0"]\arrow[r,"\M^{(0{,}p+\Z)}","\sim"']\& \finmod.\arrow[from=u,"\Phi_{p+\Z}^l"]
\end{tikzcd}
\]
\end{corollary}

\section{Proof of Theorem \ref{thm:formalCycles}}\label{subsec:RestrProof}

Throughout this section, an object of $\Loc\times_{\Locc} \Hol(\ddAp)$ will be written as a pair $M=\linebreak(M_{\pz},M_{\As})\in \Loc\times \Hol(\ddAp)$, where there is a fixed isomorphism $M_{\pz}|_{\pzz}\cong M_{\As}|_{\pzz}$. We will often omit this isomorphism and think of it as an identification $M_{\pz}|_{\pzz}= M_{\As}|_{\pzz}$.

All four restriction functors are essentially tensor products, and the adjunction between tensoring and the forgetful functor induces certain morphisms of modules. For $M\in \Hol(\ddA)$, there is a $\ddA$-module homomorphism
\[
\f{|_{\As}}{M}{M|_{\As}}
{m}{m|_\As = 1\otimes m\in \C\left[
z,\frac{1}{z-p}
\right]\otimes_{\C[z]} M.}
\]
Also, there is a countable collection of $\C[z]$-module homomorphisms
\[
\f{|_{\pzi}}{M}{M|_{\pz}}
{m}{m|_\pzi = 1\otimes \tau^{-i}m\in \C\left[\left[
\pi
\right]\right]\otimes_{\C[z]} M.}
\]
Recall that we denote $\pi=z-p$. Note that $((z-p-i)m)|_{\pzi} = \pi (m|_{\pzi})$, and $(\tau m)|_{\pzi} = m|_{U_{p+i-1}}$. Similarly, there are two more homomorphisms for the other restrictions. If $M_{\pz}\in \Mod(\C[[\pi]])$ and $M_{\As}\in \Hol(\ddAp)$, we have
\[
\begin{tikzcd}[column sep = 4 em, row sep=-0.5 em,ampersand replacement=\&]
{M_{\pz}}\arrow[r,"|_{\pzz}"] \& {\C((\pi))\otimes M_{\pz}} \&
{M_{\As}}\arrow[r,"|_{\pzzi}"] \& {M_\As|_{\pzz}} \&
\\
{m}
\arrow[r,mapsto] \&
{m|_\pzz=1\otimes m} \&
{m} 
\arrow[r,mapsto]\&{m|_\pzzi = 1\otimes \tau^{-i}m}.
\end{tikzcd}
\]


\begin{proof}[Proof of Theorem \ref{thm:formalCycles}]
Let $G$ be the induced functor $\Hol(\ddA)\to \Loc\times_{\Locc} \Hol(\ddAp)$, which is given by $G(M)= (M|_\pz,M|_\As,\cong_M )$, where $M|_\pz|_{\pzz}\cong_M M_{\As}|_{\pzz}$ is the isomorphism from Proposition \ref{prop:squareCommutes}. The claim is that $G$ is an equivalence.

Let us construct an inverse. Consider an object in $\Loc\times_{\Locc} \Hol(\ddAp)$. It is of the form $M=(M_{\pz},M_{\As})$, where $M_{\pz}\in \Loc$, so it has two distinguished submodules $M^l_{\pz},M^r_{\pz}\subset M_{\pz}$, and $M_{\As}\in \Hol(\ddAp)$. The final piece of the data is an isomorphism $M_{\pz}|_{\pzz}\cong M_\As|_{\pzz}$, identifying $M_{\pz}|_{\pzz}^{lr}$ with $M_\As|_{\pzz}^{lr}$. We will identify these two objects in $\Locc$ and call them both $M|_{\pzz}$. We construct the following module:
\[
F(M):=
 \{
\left( (m_i)_{i\in \Z},m_\As\right)\in M_\pz^{\Z}\oplus M_\As:
m_\As|_\pzzi = m_i|_{\pzz};
m_i\in M_\pz^l\text{ for }i\ll 0;
m_i\in M^r_\pz\text{ for }i\gg 0
\}
.\]
It has the structure of a $\ddA$-module in the following way:
\[\begin{array}{rclcrcl}
z((m_i)_i,m_\As ) &=&(((\pi+p+i)m_i)_i,zm_\As) & ;
&\tau ((m_i)_i,m_\As ) &=&((m_{i-1})_i,\tau m_\As).
\end{array}
\]
One can check that $z$ and $\tau$ preserve $F(M)$, and that $z\tau=\tau(z-1)$, so indeed $F(M)$ is a $\ddA$-module. 
The map $F$ can be made into a functor in the following way: A morphism $f:(M_\pz,M_\As)\to (N_\pz,N_\As)$ consists of a pair of morphisms $ f_\pz:M_\pz\to N_\pz$ and $ f_\As:M_\As\to N|_\As$ such that $f_\pz|_\pzz = f_\As|_\pzz$, i.e. the following diagram commutes:
\[
\begin{tikzcd}[column sep = 5 em,ampersand replacement=\&]
M_\pz|_\pzz \arrow[r,"f_\pz|_\pzz"]\arrow[d,"\cong_M", leftrightarrow] \& N_\pz|_\pzz\arrow[d,"\cong_N", leftrightarrow] \\
M_\As|_\pzz \arrow[r,"f_\As|_\pzz"] \& N_\As|_\pzz.\\
\end{tikzcd}
\]\vspace{-2em}

So we can identify both horizontal arrows as one map $f|_\pzz:M|_\pzz\to N|_\pzz$. We define a map $Ff:F(M)\to F(N)$, given by
\[
F(f):((m_i)_{i},m_\As)\longmapsto ((f_\pz m_i)_{i},f_{\As}m_\As).
\]

Before we prove that $F$ is well-defined (i.e. that $F(M)$ is holonomic), let us prove a useful lemma.
\begin{lemma}\label{lem:TorsioninF(M)}
For a module $M\in \Loc\times_\Locc \Hol(\ddAp)$, let $K(M)\subset F(M)$ be defined as the sub-$\ddA$-module of $F(M)$ consisting of sections supported on $p+\Z$. Then
\[
K(M) = F(M)\cap (M_\pz^{\Z}\oplus 0) \subset F(M).
\]
Further, $K(M)$ is generated over $\ddA$ by the elements of $M_\pz^{\Z}\oplus M_\As$ of the form
\[\{
((m_i)_i,m_\As)\in M_\pz^{\Z}\oplus M_\As:m_\As = 0; m_i =0 \forall i\neq 0
\} = \{m\in K(M):\exists N, (z-p)^Nm = 0 \} = K(M)_p.
\]
Let $\ov{F(M)} = \displaystyle\frac{F(M)}{K(M)}$, and let $\ov M_\pz$ be the image of $M_\pz$ in $M_\pzz$. Then
\[
\ov{F(M)} \cong \{m \in  M_\As:m|_\pzzi\in \ov M_\pz \forall i\}.
\]
\end{lemma}
\begin{proof}
Suppose $m=((m_i)_i,m_\As)\in F(M)$ is supported on $p+\Z$. This means that there is some $P(z)$ such that $P(z)m=0$ whose roots are contained in $p+\Z$. In particular, $P(z)m_\As=0$, and since $P(z)$ is a unit in $\C\left[z,\right\{\frac{1}{z-p-i}\left\}_i \right]$, we have that $m_\As =0$. The fact that $ m_i|_{\pzz} =m_\As|_\pzz= 0$ implies that every element $m_i$ is torsion. Since $M_\pz^l$ and $M_\pz^r$ are torsion free, $m_i=0$ for $|i|\gg 0$. Therefore, $K(M)$ is generated by the elements $((m_i),m_\As)$ for which $m_0$ is the only nonzero entry, as desired.

From the fact that $K(M) = F(M)\cap (M_\pz^{\Z}\oplus 0)$, we have that $\ov{F(M)}$ injects into $M_\As$, and we have that
\begin{align*}
\ov{F(M)} &\cong \{
m\in M_\As: \exists m_i\in M_\pz,m|_{\pzzi} =m_i|_{\pzz}; m_i\in M_\pz^l\text{ for }i\ll 0;m_i\in M_\pz^r\text{ for }i\gg 0
\} .
\end{align*}
Since $M_\pz^{lr}$ are torsion-free, they map isomorphically into their images $M_\pz|_\pzz^{lr}=M_\As|_\pzz^{lr}\subset \ov M_\pz$, which means $\ov{F(M)}$ can be seen as
\begin{align*}
\ov{F(M)}&\cong \{
m \in  M_\As:m|_\pzzi\in \ov M_\pz;m|_\pzzi\in \ov M_\pz^l\text{ for }i\ll 0;m|_\pzzi\in \ov M_\pz^r\text{ for }i\gg 0
\}.\end{align*}
Now we observe that the last two conditions are vacuous, since any $m\in M_\As$ is contained in an element of $\Gl(M_\As)$. Therefore,
\begin{align*}
\ov{F(M)}&\cong\{
m \in  M_\As:m|_\pzzi\in \ov M_\pz 
\}.
\end{align*}
\end{proof}

\begin{lemma}\label{lem:GinvWellDefd}
The construction of $F$ above defines a functor $F:\Loc\times_{\Locc} \Hol(\ddAp)\to \Hol(\ddA)$.
\end{lemma}
\begin{proof}
Let us first show that for $M=(M_\pz,M_\As)$, $F(M)$ is holonomic. Consider the exact sequence
\begin{equation}\label{eq:sesF(M)}
0\longrightarrow K(M) \longrightarrow F(M) \longrightarrow M_\As.
\end{equation}
We use Lemma \ref{lem:TorsioninF(M)}: since $M_\pz$ is finitely generated and $K(M)_p\subset M_\pz$, $K(M)_p$ is a finitely generated $\C[z]$-module, so $K(M)$ is finitely generated over $\ddA$. Its generators are torsion over $\C[z]$, so in particular they are torsion over $\ddA$. This proves that $K(M)$ is holonomic.

Therefore, it suffices to show that $\ov{F(M)}=\frac{F(M)}{K(M)}\subset M_\As$ is holonomic. Since it is contained in the holonomic $\ddAp$-module $M_\As$, every element is torsion, so we only have to prove that it is finitely generated.

Using Lemma \ref{lem:TorsioninF(M)}, we have that $\ov{F(M)}\cong  \{
m \in  M_\As:m_\pzzi\in \ov M_\pz 
\}$, and it remains to prove that such a module is finitely generated over $\ddA$. First of all, choose some $L\in \Gl({\ov{F(M)}})$ (this is possible by Observation~\ref{obs:glattices}). Then $L$ is finitely generated over $\C[z]$, and ${\ov{F(M)}}/L$ is a torsion module. For a given $i$, ${\ov{F(M)}}_{p+i}\cong \ov M_\pz$ is a finitely generated $\C[[z-p-i]]$-module, and therefore its torsion quotient $({\ov{F(M)}}/L)_{p+i}$ is finitely generated over $\C[z]$. Now we note that for big enough $i$, $\tau$ induces isomorphisms $L_{p-i-1}\to L_{p-i}$ and $L_{p+i}\to L_{p+i+1}$, and therefore a finite collection of elements of ${\ov{F(M)}}$ suffice to generate $({\ov{F(M)}}/L)_{p+i}$ for all $i$, over $\ddA$. Putting everything together, ${\ov{F(M)}}$ is indeed finitely generated over $\ddA$.

Let us show that the action of $F$ on morphisms is well-defined. Let $f:M\to N$ be a morphism in $\Loc\times_\Locc \Hol(\ddAp)$, and consider $F(f):F(M)\to F(N)$, defined as above by $f((m_i)_{i},m_\As)=((f_\pz m_i)_{i},f_{\As}m_\As)$. We claim that this map is well-defined. It is straightforward to check that $F(f)$ is a $\ddA$-module homomorphism. We will now prove that its image is contained in $F(N)$. Let us show that $(f_{\As}m_\As)|_\pzzi =(f_\pz m_i)|_\pzz$:
\[
(f_{\As}m_\As)|_\pzzi =(f_{\As}\tau^{-i}m_\As)|_\pzz =f|_{\pzz}(\tau^{-i}m_\As)|_\pzz =f|_{\pzz}(m_\As)|_\pzzi =f|_{\pzz}(m_i)|_\pzz =(f_{\pz}m_i)|_\pzz 
.\]
Next we have to show that $f_\pz m_i \in N_\pz^l$ for $i\ll 0$: if $i\ll 0$, $m_i\in M_\pz^l$, so $f_\pz m_i\in N_\pz^l$. Since $f_\pz$ is a morphism in $\Loc$, it maps $M_\pz^l$ into $N_\pz^l$. Similarly, it can be shown that $f_\pz m_i \in N_\pz^r$ for $i\gg 0$.

Given that $F$ is well-defined, it is clear that indeed it is a functor, i.e. that it preserves compositions and it maps identity morphisms to identity morphisms.
\end{proof}
\begin{lemma}
$F\circ G \cong \Id_{\Hol(\ddA)}$.
\end{lemma}
\begin{proof}
Let $M\in \Hol(\ddA)$. There is a natural map $\phi :M\to F(G(M))$, given by
\[
\phi(m) = ((m|_\pzi)_i,m|_\As)
.\]
It can be checked that $\phi(m)\in F(G(M))$: this amounts to showing that $m|_\pzi|_\pzz = m|_\As|_\pzzi$, which is Proposition \ref{prop:squareCommutes}, and that $m|_\pzi\in M_\pz^{lr}$ for big enough or small enough $i$. This follows from the fact that any element is contained in an element of $\Gl(M)$. So indeed $\phi$ is well-defined. We claim that $\phi$ is injective: if $\phi(m)=0$, then $m|_\As=0$, which implies that $m$ is supported on $p+\Z$, and if $m|_\pzi=0$ for all $i$, then $m$ has no support on any point of $p+\Z$ either. Therefore, $\phi$ is injective.

Now let us show that $\phi$ is surjective. Consider the sequence (\ref{eq:sesF(M)}) applied to $G(M)=(M|_\pz,M|_\As)$:
\[
0\longrightarrow K(G(M)) \longrightarrow F(G(M)) \longrightarrow G(M)_\As
.\]
Let $\ov{F(G(M))}$ be the image of $F(G(M))$ in $G(M)_\As = M|_\As$, so we have a short exact sequence
\[
0\longrightarrow K(G(M)) \longrightarrow F(G(M)) \longrightarrow \ov{F(G(M))} \longrightarrow 0
.\]
The composition $M\overset{\phi}\to F(G(M))\to \ov{F(G(M))}\subset M|_\As$ is just the natural map $|_\As$. We claim that $M\to\ov{F(G(M))}$ is surjective, which boils down to
\[
\ov M := \{m\in M|_\As :\exists m'\in M,m=m'|_\As \} \supseteq \{m\in M|_\As :m|_\pzzi \in \ov {M|_\pz}\forall i \} =\ov{F(G(M))}
.\]
Take an $m\in M|_{\As}$ contained in the right-hand side. By the definition of $M|_\As$, there is some $P(z)$ with roots contained in $p+\Z$ such that $P(z)m\in \ov M$. Thinking of $\ov M$ as a quasicoherent sheaf on $\A^1$, this is saying that $m$ is a section of $\ov M$ on the open set which is the complement of the roots of $P$. The fact that $m|_\pzzi\in \ov{M|_\pzi}$ implies that this section is regular at the points which are roots of $P(z)$. Since $\ov M$ is a sheaf, this means that $m$ is a global section of $\ov M$, as we wished to prove.

Therefore, $\phi$ induces a surjection onto $\ov{F(G(M))}$, so we just have to show that the image of $\phi$ contains $K(G(M))$. By Lemma \ref{lem:TorsioninF(M)}, $K(G(M))$ is generated by $K(G(M))_p=\{(m_i)\in K(G(M)):m_i=0 \forall i\neq 0 \}$. All of these elements are in the  image of $\phi$, since they are exactly the image of the elements of $M$ whose support is $\{p\}$.
\end{proof}

\begin{lemma}
$G\circ F \cong \Id_{\Loc\times_{\Locc}\Hol(\ddAp)}$.
\end{lemma}
\begin{proof} Let $M=(M_\pz,M_\As)\in \Loc\times_\Locc \Hol(\ddA)$. Let us construct a natural map $\psi:G(F(M))\to M$. Let us write $G(F(M))=(F(M)|_\pz,F(M)|_\As)$. We define
\[\begin{array}{rcclcrccl}
\psi_\pz: & F(M)|_\pz &\longrightarrow & M_\pz & \text{;}& \psi_\As: & F(M)|_\As & \longrightarrow & M_\As \\
     & ((m_i)_i,m_\As)|_\pz& \longmapsto & m_0 & & & ((m_i)_i,m_\As)|_\As & \longmapsto & m_\As.
\end{array}
\]
We must prove all of the following.
\begin{enumerate}
\item $\psi_\pz$ does not depend on the choice of a representative, i.e. if $((m_i)_i,m_\As)|_\pz=((m_i')_i,m_\As')|_\pz$, then $m_0 = m'_0$. Similarly, $\psi_\As$ doesn't depend on the choice of a representative.
\item $\psi_\pz$ is $\C[[\pi]]$-linear and $\psi_\As$ is $\ddAp$-linear.
\item $\psi_\pz$ maps $F(M)|_\pz^{lr}$ into $M_\pz^{lr}$.
\item $\psi_\pz|_\pzz = \psi_\As|_\pzz$.
\item $\psi_\pz$ and $\psi_\As$ are bijections.
\item $\psi_\pz$ induces bijections $F(M)|_\pz^{lr}\overset{\sim}\to M_\pz^{lr}$.
\end{enumerate}
From the above 6 statements, it follows that $(\psi_\pz,\psi_\As)$ is an isomorphism. Here are the proofs.
\begin{enumerate}
\item By the fact that tensoring is a left adjoint, the map $\psi_\pz:F(M)|_\pz\to M_\pz$ is equivalent to a $\C[\pi]$-linear map $F(M)\to M_\pz$, namely the map $((m_i)_i,m_\As)\mapsto m_0$, which gives rise to $\psi_\pz$ after tensoring. Similarly, $\psi_\As$ comes from the $\C[z]$-linear map $((m_i)_i,m_\As)\mapsto m_\As$, after tensoring by $\C[z,(z-p)^{-1}]$.

\item This is clear given that $\psi_\pz$ and $\psi_\As$ are well defined.
\item Let $L_{F(M)}\in \Gl(F(M))$, and let $S$ be a finite generating set for it. By definition, $\tau^N S$ generates $F(M)|_\pz^l$ for $N\gg 0$. Since $S$ is finite, there is an $N$ such that if $i\ge N$ and for any $((m_i)_i,m_\As)\in S$, $m_{-i}\in M_\pz^l$. Therefore, picking $N$ big enough, for $m = ((m_i)_i,m_\As)\in S$,
\[
\psi_\pz \tau^Nm = \psi_\pz ((m_{i-N})_i,\tau^Nm_\As) =m_{-N}\in M_\pz^l 
.\]
So the generators of $F(M)|_\pz^l$ are mapped into $M_\pz^l$. The analogous proof shows $\psi_\pz(F(M))|_\pz^r \subseteq M_\pz^r$.
\item The statement $\psi_\pz|_\pzz = \psi_\As|_\pzz$ amounts to saying that for $ ((m_i)_i,m_\As)\in F(M)$, $m_0|_\pzz = m_\As|_\pzz$, which is true by the definition of $F(M)$.
\item Let us first show that $\psi_\pz$ is injective. Suppose $m=((m_i)_i,m_\As)\in F(M)$ and $\psi_\pz(m|_\pz)=0$, i.e. $m_0=0$. Therefore, $m_\As|_\pzz = m_0|_\pzz =0$, which implies that the support of $m_\As$ is finite, so it is annihilated by some nonzero $P(z)\in \C\left[ z,\left\{\frac{1}{z-p-i}\right\}_i\right]$, which can be multiplied by a unit to make it a polynomial with no roots in $p+\Z$. Consider now $m'=P(z)m$. The fact that $m'_\As=0$ implies that  $m'\in K(M)$, and therefore it is annihilated by a polynomial $Q(z)\in \C[z]$ whose roots are contained in $p+\Z$ and $Q(p)\neq 0$, since $m'_0=0$. Since $Q$ is a unit in $\C[[z-p]]$, this means that $m'|_\pz=0$, and since $P(z)$ is also a unit in $\C[[z-p]]$, $m|_\pz=0$. If $\psi_\pz$ had nonzero kernel, then the kernel would contain some nonzero element of the form $m|_\pz$, for some $m\in F(M)$, which as we have just shown cannot happen, so we indeed have that $\psi_\pz$ is injective.

Let us now show that $\psi_\As$ is injective. Suppose $m=((m_i)_i,m_\As)\in F(M)$ and $\psi_\As(m|_\As)=0$, i.e. $m_\As=0$. Therefore, $m\in K(M)$. By Lemma \ref{lem:TorsioninF(M)}, $m$ is annihilated by a unit in $\C\left[z,\left\{\frac{1}{z-p-i}\right\}_i\right]$. This implies that $m|_\As$ is annihilated by a unit in the ring, so $m|_\As=0$. As before, if $\psi_\As$ had a kernel, it would intersect the image of $F(M)$ in $F(M)|_\As$. Therefore $\psi_\As$ is injective.

Let us show that $\psi_\pz$ is surjective. Let $\eta\in M_\pz$. First, suppose $\eta$ is torsion. Then, there's an element $m$ of $F(M)$ given by $m_i=0$ for $i\neq 0$, $m_0=\eta$, and $m_\As=0$, so $\psi_\pz(m|_\pz)=\eta$ as desired. If we now let $t:M_\pz\to \ov M_\pz$ be the quotient of $M_\pz$ by its torsion, it suffices to prove that the composition $t\circ \psi_\pz$ is surjective. We have that $\ov M_\pz\subset \C((\pi))\otimes M_\pz = M|_\pzz$. We claim that $\ov M_\pz$ is generated by elements of the form $m|_\pzz$, where $m\in M_\As$.

Let $n$ be the generic rank of $M_\As$, and let $\{e_1,\ldots ,e_n\}\subset M_\As$ be a $\C(z)$-basis of $\C(z)\otimes M_\As$. Let $\{e_1',\ldots ,e_n'\}$ be a $\C[[\pi]]$-basis of $\ov M_\pz$. Then they are both $\C((\pi))$-bases of $M_\pzz$, so there is some matrix $B$ such that $e'_i=Be_i$, where $B\in \GL_n(\C((\pi)))$. Now we use the following lemma, which is proved at the end of the section.

\begin{lemma}\label{lem:matricesOverSeries}
\[
\GL_n(\C((\pi)))=\GL_n(\C[[\pi]])\GL_n(\C[\pi,\pi^{-1}])
.\]
\end{lemma}

If we write $B=AC$, with $C\in \GL_n\left(\C\left[z,\frac{1}{z-p}\right]\right)$ and $A\in \GL_n(\C[[\pi]])$, we have that $\{Ce_i \}$ is another $\C(z)$-basis of $\C(z)\otimes M_\As$, and $\{ A^{-1}e_i'\}$ is another $C[[\pi]]$-basis of $\ov M_\pz$. By the identity $B=AC$, $Ce_i=A^{-1}e'_i$, so $\ov M_\pz$ is generated by elements in the image of $M_\As$. Let $\ov m_0$ be one of these elements, which we want to prove are in the image of $t\circ \psi_\pz$. Since it's in the image of some element $m_\As\in M_\As$, we may consider elements $\ov m_i = m_\As|_\pzzi\in \ov M_\pz$. If we find $m_i\in M_\pz$ with $m_i|_\pzz = \ov m_i$, $m_i\in M_\pz^l$ for $i\ll 0$ and $m_i\in M_\pz^r$ for $i\gg 0$, then we will have that
\[
t\psi_\pz((m_i)_i,m_\As) = tm_0=\ov m_0
.\]
As desired, proving that $\ov m_0$ is in the image of $t\circ \psi_\pz$. In order to prove that the $m_i$'s exist, observe that $t:M_\pz\to \ov M_\pz$ is surjective, so there exist some $m_i$'s such that $m_i|_\pzz=\ov m_i$. Since $\ov m_i= m_\As|_\pzzi$, we have that $\ov m_i\in M_\pzz^l$ for $i\ll 0$. Since $t$ induces an isomorphism $M_\pz^l\to M_\pzz^l\subset \ov M_\pz$, we may choose $m_i$ in $M_\pz^l$ for $i$ small enough, and similarly $m_i\in M_\pz^r$ for $i$ big enough. This finishes the proof that $\psi_\pz$ is surjective.

Finally, let us show that $\psi_\As$ is surjective. Let $m_\As\in M_\As$, and consider the sequence $(\tau^{-i} m_\As)|_\pzz\in M_\As|_\pzz$: by definition of $|_\pzz:\Hol(\ddAp)\to \Locc$, for $i\ll 0$ we have $(\tau^{-i} m_\As)|_\pzz\in M_\As|_\pzz^l$ and for $i\gg 0$, $(\tau^{-i} m_\As)|_\pzz\in M_\As|_\pzz^r$. This implies that for almost all $i$'s, $m_\As|_\pzzi\in M_\pzz^l\cup M_\pzz^r\subset \ov M_\pz$. Therefore, for some polynomial $P(z)$ with roots in $p+\Z$, we have that $P(z)m_\As|_\pzzi \in \ov M_\pz$, where $\ov M_\pz$ is the image of $M_\pz$ in $M_\pz|_\pzz = M_\As|_\pzz$. Let $\ov m_i = P(z)m_\As|_\pzzi \in \ov M_\pz$.

We must find a sequence $m_i\in M_\pz$ with $m_i|_\pzz = \ov m_i$, $m_i\in M_\pz^l$ for $i\ll 0$ and $m_i\in M_\pz^r$ for $i\gg 0$. In order to do this, we may proceed as before, using the fact that $t:M_\pz\to \ov M_\pz$ is surjective and it induces isomorphisms $M_\pz^l\to M_\pzz^l\subset \ov M_\pz$. Then $\psi_{\As}((m_i)_i,P(z)m_\As)=P(z)m_\As $. Since $P(z)$ is a unit, this implies that $m_\As$ is in the image of $m_\As$, as desired.

\item It only remains to prove that the restriction $\psi_\pz^l:F(M)|_\pz^l\longrightarrow M_\pz^l$ is surjective (and the proof for $\psi_\pz^r$ will be analogous). Given that both modules are torsion-free, and that we already know that $\psi_\pz$ is a bijection, we may kill all the torsion, consider instead the map $\ov{F(M)|_\pz^l}\to \ov M_\pz^l$ and show that it is surjective. We also have that $\ov{F(M)|_\pz} \cong \ov{F(M)}|_\pz$, where $\ov{F(M)}$ is the quotient of $F(M)$ by the elements supported on $p+\Z$. By Lemma \ref{lem:TorsioninF(M)}, $\ov{F(M)}\cong \{m\in M_\As:m|_\pzzi\in \ov M|_\pz\forall i\}\cong F(\ov M_\pz,M_\As)$.

Therefore, it suffices to show that the map $\ov\psi:G(F(\ov{M}))\to \ov M = (\ov M_\pz,M|_\As)$ is an isomorphism in $\Loc\times_\Locc \Hol(\ddAp)$. We have already shown that both its components $\ov \psi_\pz$ and $\ov \psi_\As$ are bijections. It only remains to observe that for an element $N\in \Loc\times_\Locc \Hol(\ddAp)$ such that $N_\pz$ is torsion-free, the modules $N_\pz^{lr}$ are determined by $N_\pz$ and $N_\As$, just by the fact that $N_\pz\to N_\pz|_\pzz$ is an injective map. Therefore, $\ov\psi$ must indeed be an isomorphism.

\end{enumerate}

We have thus proven that $(\psi_\pz,\psi_\As)$ is an isomorphism $G\circ F\cong \Id_{\Loc\times_\Locc \Hol(\ddAp)}$. 
\end{proof}

The last two lemmas together prove that $G$ and $F$ are mutual inverses, and therefore $G$ is an equivalence, as desired.

\end{proof}

\begin{proof}[Proof of Lemma \ref{lem:matricesOverSeries}]
Let $A\in \GL_n(\C((\pi)))$. We show a sequence of row and column operations with coefficients in $\C[[\pi]]$ and $\C[\pi,\pi^{-1}]$ respectively yield the identity matrix.

First, row operations with coefficients in $\C[[\pi]]$ allow to make the matrix upper triangular. Then, multiplying by diagonal matrices on the left and on the right (with the correct coefficients) can ensure that the coefficients along the diagonal are all $1$. At this point, more row operations can ensure all the remaining nonzero coefficients become Laurent polynomials, so the resulting matrix is in $\GL_n(\C[\pi,\pi^{-1}])$.

\end{proof}

\bibliographystyle{alpha}
\bibliography{Bibliography}

\end{document}